\documentclass[a4paper,11pt]{article}

\usepackage[margin=1in]{geometry}
\usepackage{setspace}
\usepackage{amsmath}
\usepackage{amssymb}
\usepackage{amsthm}
\usepackage{graphicx}
\usepackage{float}
\usepackage{caption}
\usepackage{listings}
\usepackage{tikz}
\usepackage{enumitem}
\usepackage{xcolor}

\usepackage{mathrsfs}

\newtheorem{theorem}{Theorem}[section]
\newtheorem{proposition}[theorem]{Proposition}
\newtheorem{corollary}[theorem]{Corollary}

\newtheorem{lemma}[theorem]{Lemma}
\newtheorem{definition}[theorem]{Definition}

\newtheorem{remark}[theorem]{Remark}

\numberwithin{figure}{section}

\providecommand{\keywords}[1]{\textbf{Keywords:} #1}

\usepackage{hyperref}
\hypersetup{
    hidelinks,
    pdftitle={A Self-Dual Frame Formalism of the SO(3) Yang-Mills Theory},
    pdfauthor={Hanwen Liu}
}

\begin{document}

\title{\texorpdfstring{\textbf{A Self-Dual Frame Formalism of the SO(3) Yang-Mills Theory}}{A Self-Dual Frame Formalism of the SO(3) Yang-Mills Theory}}
\author{Hanwen Liu\thanks{ORCID: 0009-0007-2503-9576}\\
{\small Mathematics Institute, University of Warwick}\thanks{University of Warwick, Coventry, United Kingdom, CV4 7AL}\\
{\small \href{mailto:hanwen.liu@warwick.ac.uk}{\texttt{hanwen.liu@warwick.ac.uk}}}}
\date{}
\maketitle

\begin{abstract}
For a closed oriented Riemannian $4$-manifold $(M,g)$, we consider $\operatorname{SO}(3)$ connections on the bundle $\Lambda^+$ of self-dual $2$-forms. On the open locus where the self-dual curvature is an orientation-preserving frame, pointwise polar decomposition removes the gauge freedom and replaces the connection by a positive self-adjoint endomorphism field $h$ of $\Lambda^+$. Using the classical reconstruction of the compatible connection $A(h)$, we obtain a global formulation of the Yang--Mills equation as the determined second order system
$$
\Phi_g(h):=F_{A(h)}^+h^{-1}-\operatorname{Id}=0.
$$
Here, the tensor $F_{A(h)}^+$ is the self-dual curvature of $A(h)$, regarded as an endomorphism of $\Lambda^+$, and $\operatorname{Id}$ is the identity of $\Lambda^+$. We establish a variational formulation, automatic irreducibility, elliptic regularity, and a Fredholm index theorem for the linearized operator $D_h\Phi_g$. For fields of the form $h=e^{2\omega}\operatorname{Id}$, where $\omega$ is a smooth real-valued function on $M$, the equation $\Phi_g(h)=0$ is equivalent to anti-self-duality and constant scalar curvature $6\sqrt{2}$ of the conformal metric $\hat g=e^{2\omega}g$. Consequently, every anti-self-dual conformal class of positive Yamabe constant gives a global solution of $\Phi_g(h)=0$.
\end{abstract}

\begin{center}
\keywords{Yang--Mills field; Yamabe problem; gauge theory}
\end{center}

\tableofcontents
\onehalfspacing
\raggedbottom

\section{Introduction and Background}

The Yang--Mills equations were introduced by Yang and Mills as the Euler--Lagrange equations of a non-Abelian gauge theory \cite{yangmills}. In dimension four, the decomposition of $2$-forms into self-dual and anti-self-dual parts gives the theory a special structure. The first order instanton equations imply the full Yang--Mills equation, and their moduli spaces have become fundamental objects in four-dimensional geometry and topology \cite{bpst,ahs,donaldson,freeduhlenbeck,donaldsonkronheimer}. By contrast, a general Yang--Mills connection satisfies a second order equation with gauge degeneracy. Its analysis usually proceeds by choosing a gauge, studying the corresponding Jacobi operator, and applying compactness and removable-singularity theory \cite{bourguignonlawson,uhlenbeckLp,uhlenbeckrem}. The full Yang--Mills equation also admits a formally self-adjoint elliptic deformation complex in dimension four \cite{gover}.

The purpose of this article is to give a global formulation of an open sector of the full $\operatorname{SO}(3)$ Yang--Mills equation on a closed oriented Riemannian $4$-manifold $(M,g)$. Let $A$ be an $\operatorname{SO}(3)$ connection on an oriented Euclidean vector bundle $\mathcal{E}\rightarrow M$ of rank three. Using the metric and orientation of $\mathcal{E}$ to identify $\mathfrak{so}(\mathcal{E})$ with $\mathcal{E}$, we regard the self-dual curvature of $A$ as a homomorphism $F_A^+\colon\mathcal{E}\rightarrow\Lambda^+$. Here, the homomorphism $F_A^+$ is obtained by contracting the $\mathcal{E}$-valued factor of the curvature using the metric of $\mathcal{E}$, as specified in equation~(\ref{curvature_homomorphism}) below.

Assume that the curvature homomorphism $F_A^+$ is an orientation-preserving isomorphism at every point of $M$. The polar decomposition of $F_A^+$ determines a unique orientation-preserving bundle isometry $R\colon\mathcal{E}\rightarrow\Lambda^+$ such that the self-dual curvature of the transported connection $R_*A$ is a positive self-adjoint endomorphism $h$ of $\Lambda^+$. If $\mathcal{E}=\Lambda^+$, then $R$ is an automorphism of the fixed oriented Euclidean bundle $\Lambda^+$, and therefore a gauge transformation. Thus, when $\mathcal{E}=\Lambda^+$, the connection $R_*A$ is the representative of the gauge class of $A$ whose self-dual curvature is the positive self-adjoint endomorphism $h$ of $\Lambda^+$. In transporting $A$ by $R$, we keep the metric $g$ and the ordinary self-dual $2$-forms on $M$ fixed. Theorem~\ref{Yang_Mills_equivalence} gives the precise correspondence between positive self-adjoint endomorphism fields $h$ of $\Lambda^+$ satisfying $\Phi_g(h)=0$ and gauge classes of Yang--Mills connections whose self-dual curvature is an orientation-preserving isomorphism.

The reconstruction of a connection from a non-degenerate triple of ordinary $2$-forms is classical. It appears in the work of Deser--Teitelboim \cite{deserteitelboim} and in Halpern's field-strength formulation of non-Abelian gauge theory \cite{halpern}. Freidel \cite[Section~II.B]{freidel} explicitly reconstructs the compatible connection of a non-degenerate triple of $2$-forms. In the present article, a positive self-adjoint endomorphism $h$ of $\Lambda^+$ determines a $\Lambda^+$-valued self-dual $2$-form $B_h$ by equation~(\ref{form_from_endomorphism}). The corresponding bundle homomorphism
$$
\operatorname{ad}_h\colon T^*M\otimes\Lambda^+\rightarrow\bigwedge^3T^*M\otimes\Lambda^+
$$
is defined by $\operatorname{ad}_h(a)=[a,B_h]$ for every $\Lambda^+$-valued $1$-form $a$, where the bracket combines the exterior product with the Lie bracket on $\Lambda^+$. We prove that $\operatorname{ad}_h$ is invertible. Consequently, the equation $d_AB_h=0$ determines a unique $\operatorname{SO}(3)$ connection $A(h)$ on $\Lambda^+$. The connection $A(h)$ is the restriction of the reconstruction in \cite{freidel} to triples which are self-dual for the fixed metric $g$.

There is also a direct earlier treatment of the Yang--Mills equation in these variables. Halpern \cite[Section~V]{halpernselfdual}, using the observation of Corrigan--Fairlie--Yates, describes the non-degenerate self-dual curvature by six gauge-invariant variables, without requiring the anti-self-dual curvature to vanish. Thus, neither the reconstruction nor the use of six variables for non-instantonic fields is new. The present contribution is a global formulation of these curvature variables by positive self-adjoint endomorphism fields of $\Lambda^+$ on a closed Riemannian $4$-manifold, together with the principal-symbol calculation, elliptic regularity, the Fredholm index theorem, and the conformal reduction in Section~2.4.

The relation between $h$ and the triple of $2$-forms used in \cite{freidel} can be stated directly. Choose an open subset $U\subseteq M$ and a local oriented orthonormal frame $(\Sigma_1,\Sigma_2,\Sigma_3)$ of $\Lambda^+|_{U}$. Let $$h_{ij}:=\langle h(\Sigma_j),\Sigma_i\rangle$$ be the matrix entries of $h$ in this frame. The ordinary self-dual $2$-forms $B_h^1,B_h^2,B_h^3$ are the coefficients of the $\Lambda^+$-valued $2$-form $B_h$ in the frame $(\Sigma_1,\Sigma_2,\Sigma_3)$ of the coefficient bundle. Since $h$ is symmetric, these coefficient forms satisfy
$$
B_h^i=\sum_{j=1}^3h_{ij}\Sigma_j.
$$
The matrix defined by the wedge products of the coefficient forms is $h^2$; more precisely, for all $i,j\in\{1,2,3\}$ we have
$$
B_h^i\wedge B_h^j=(h^2)_{ij}d\mu_g,
$$
where $d\mu_g$ is the volume form of $(M,g)$. Thus, the matrix of our endomorphism field $h$ in the frame $(\Sigma_1,\Sigma_2,\Sigma_3)$ is the positive coefficient matrix of the triple $(B_h^1,B_h^2,B_h^3)$, whereas Freidel's internal metric is obtained from the wedge products of that triple after the volume normalization used in \cite{freidel}. We keep the ordinary forms $\Sigma_j$ fixed, whereas the construction in \cite{freidel} also allows the conformal structure determined by the triple to vary.

The Plebanski and pure-connection formulations of gravity \cite{plebanski,capovilla,fine} impose different curvature equations. In the present article, the equation $F_{A(h)}^+=h$ specifies the self-dual curvature of $A(h)$ and permits non-zero anti-self-dual curvature $F_{A(h)}^-$. The reduced functional $E[h]$ in Proposition~\ref{E-L} is obtained from the BF action with a quadratic cosmological term. We first restrict the independent $2$-form variable of that action to forms which are self-dual for $g$, choose the positive symmetric representative $B_h$, and then eliminate the connection by solving $d_AB_h=0$.

The main results are as follows. We first recall the construction of $A(h)$ and prove that $F_{A(h)}^+h^{-1}$ is a self-adjoint endomorphism of $\Lambda^+$. The self-adjointness of $F_{A(h)}^+h^{-1}$ implies that $\Phi_g$ takes positive self-adjoint endomorphism fields of $\Lambda^+$ to sections of $\operatorname{Sym}^2(\Lambda^+)$, which is a bundle of rank six. We show that $\Phi_g(h)=0$ is the Euler--Lagrange equation of the reduced functional $E[h]$, and that the zero set of $\Phi_g$ is naturally identified with the gauge classes of Yang--Mills connections for which $F_A^+$ is an orientation-preserving frame. Every Yang--Mills connection whose self-dual curvature is an orientation-preserving frame is automatically irreducible and has trivial stabilizer. We then compute the linearization $D_h\Phi_g$ and its principal symbol. For every point $x\in M$ and every non-zero covector $\xi\in T_x^*M$, the principal symbol of $D_h\Phi_g$ is an automorphism of the six-dimensional vector space $\operatorname{Sym}^2(\Lambda_x^+)$. The linearized operator $D_h\Phi_g$ is Fredholm of index zero between the Sobolev spaces specified in Theorem~\ref{Elliptic_index}. In particular, the gauge reduction does not leave a residual infinitesimal gauge degeneracy.

The positive scalar multiples of the identity endomorphism of $\Lambda^+$ have a direct geometric meaning. For a smooth real-valued function $\omega$ on $M$, let $h:=e^{2\omega}\operatorname{Id}$. Then $A(h)$ is the connection induced by the Levi--Civita connection of $\hat g=e^{2\omega}g$, transported to the fixed Euclidean bundle $\Lambda^+$. The standard curvature decomposition \cite{ahs,besse} identifies $\Phi_g(h)=0$ with anti-self-duality and constant scalar curvature $6\sqrt{2}$ of $\hat g$, with the orientation and curvature conventions specified in Section~2. The solution of the Yamabe problem \cite{yamabe,trudinger,aubin,schoen} therefore gives a global existence theorem for $\Phi_g(h)=0$ on every anti-self-dual conformal $4$-manifold of positive Yamabe type. The construction of $A(e^{2\omega}\operatorname{Id})$ from the Levi--Civita connection of $e^{2\omega}g$ is related to the conformal spin-connection construction of instantons used by Etesi--Hausel \cite{etesihausel}. In our setting, the conformal metric $\hat g$ need not be Einstein, and the Yang--Mills connection $A(e^{2\omega}\operatorname{Id})$ need not be an instanton. Remark~\ref{Conformal_comparison} gives the precise comparison.

The present article is deliberately foundational. It does not address compactness of the full solution space of $\Phi_g(h)=0$ or behavior at the boundary where $F_A^+$ loses rank. The Fredholm index theorem for $D_h\Phi_g$ suggests a degree theory for $\Phi_g$, but such a theory requires a priori estimates excluding both ordinary Yang--Mills bubbling and degeneration of the curvature frame. These questions, together with the comparison between the linearized frame operator and the standard Yang--Mills deformation complex, will be considered elsewhere.

The article is organized as follows: Section~2 constructs the self-dual frame operator $\Phi_g$, proves its variational and Yang--Mills interpretations, and develops its elliptic theory. Section~2 concludes with the conformal reduction and the Yamabe existence theorem. Section~3 records several consequences and directions for further study.

\section{The Main Results}

We now develop the formalism in four steps. We first construct the connection $A(h)$ compatible with the positive self-dual frame $B_h$, define the nonlinear operator $\Phi_g$, and establish a variational formulation of the equation $\Phi_g(h)=0$. We then identify the zero set of $\Phi_g$ with the gauge classes of Yang--Mills connections whose self-dual curvature is an orientation-preserving isomorphism. After that, we establish the ellipticity of the equation $\Phi_g(h)=0$. Finally, we examine the positive scalar multiples of the identity endomorphism of $\Lambda^+$ and obtain a global existence theorem for $\Phi_g(h)=0$ from the Yamabe problem.

Throughout this article, all differentiable manifolds are assumed to be connected.
For a vector bundle $\mathcal{E}\rightarrow M$ and a point $x\in M$, denote by $\mathcal{E}_x$ the fibre of $\mathcal{E}$ over $x$. For a section $\sigma$ of $\mathcal{E}$, write $\sigma_x:=\sigma(x)\in\mathcal{E}_x$. For vector bundles $\mathcal{E}$ and $\mathcal{F}$ over $M$ and a bundle homomorphism $T\colon\mathcal{E}\rightarrow\mathcal{F}$ covering the identity of $M$, denote by $T_x\colon\mathcal{E}_x\rightarrow\mathcal{F}_x$ its restriction to the fibre over the point $x\in M$. In particular, for an endomorphism field $h$ of $\mathcal{E}$, the notation $h_x$ means the endomorphism of $\mathcal{E}_x$ obtained by evaluating $h$ at $x$. We denote the dual bundle of $\mathcal{E}$ by $\mathcal{E}^\vee$, so that $\mathcal{E}_x^\vee$ is the vector space of real-valued linear functionals on $\mathcal{E}_x$.

Fix once and for all an integer $s\geq2$ and a real number $p>4$. For a smooth vector bundle $\mathcal{E}\rightarrow M$ and an integer $k\geq0$, denote by $\Omega^k(\mathcal{E})$ the space of smooth $\mathcal{E}$-valued differential $k$-forms on $M$. In particular, $\Omega^0(\mathcal{E})$ is the space of smooth sections of $\mathcal{E}$. Denote by $\Omega^k(U)$ the space of ordinary smooth real-valued differential $k$-forms on an open subset $U\subseteq M$. Using a smooth bundle metric and a smooth reference connection on $\mathcal{E}$, denote by $W^{s,p}(\mathcal{E})$ the completion of the smooth sections of finite Sobolev norm. On a closed manifold, the topology of $W^{s,p}(\mathcal{E})$ is independent of the auxiliary smooth metric and reference connection. We use the analogous notation $W^{r,p}(\mathcal{E})$ for any integer $r\geq0$.

We fix once and for all a closed oriented Riemannian $4$-manifold $(M,g)$. For every integer $k\geq0$, denote by $\Lambda^k:=\bigwedge^kT^*M$ the bundle of ordinary differential $k$-forms on $M$. In particular, $\Lambda^1=T^*M$, and the fibre of $\Lambda^k$ at $x\in M$ is $\Lambda_x^k=\bigwedge^kT_x^*M$. For every smooth vector bundle $\mathcal{E}\rightarrow M$, we have that $$\Omega^k(\mathcal{E})=\Omega^0(\Lambda^k\otimes\mathcal{E}).$$ Denote by $\star$ the Hodge star determined by $g$ and the orientation of $M$. The bundle $\Lambda_g^+$ consists of the $2$-forms $\sigma$ satisfying $\star\sigma=\sigma$, and the bundle $\Lambda_g^-$ consists of the $2$-forms $\sigma$ satisfying $\star\sigma=-\sigma$. Whenever there exists no ambiguity, we will write $\Lambda^\pm$ instead of $\Lambda_g^\pm$, for simplicity. Both $\Lambda^+$ and $\Lambda^-$ have rank three. We write $\langle-,-\rangle$ and $|\cdot|$ for the bundle metric and norm induced by $g$, and write $d\mu_g$ for the Riemannian volume form of $(M,g)$.

Equip $\Lambda^+$ with the Euclidean metric induced by $g$. Choose an open subset $U\subseteq M$ and a local oriented $g$-orthonormal co-frame $(\theta_0,\theta_1,\theta_2,\theta_3)$ of $T^*M|_{U}$. For every cyclic permutation $(i,j,k)$ of $(1,2,3)$, define the ordinary self-dual $2$-form
\begin{equation}\label{self_dual_basis}
\Sigma_i:=-\frac{\theta_0\wedge\theta_i+\theta_j\wedge\theta_k}{\sqrt{2}}.
\end{equation}
The forms $(\Sigma_1,\Sigma_2,\Sigma_3)$ constitute a local orthonormal frame of $\Lambda^+|_{U}$. We orient $\Lambda^+$ by declaring this frame to be positively oriented. The orientation of $\Lambda^+$ specified by equation~(\ref{self_dual_basis}) is independent of the choice of local oriented orthonormal co-frame of $T^*M$. With this orientation, the self-dual curvature of the round sphere is a positive frame \cite{fine}.

For each point $x\in M$, the metric and orientation of $\Lambda_x^+$ determine a cross product on the three-dimensional vector space $\Lambda_x^+$. We denote this cross product by $[f,v]$ for $f,v\in\Lambda_x^+$. Define the skew-adjoint endomorphism $S_f\colon\Lambda_x^+\rightarrow\Lambda_x^+$ by $S_f(v):=[f,v]$. Denote by $\mathfrak{so}(\Lambda_x^+)$ the space of skew-adjoint endomorphisms of $\Lambda_x^+$, with the commutator bracket. For any elements $A,B\in\mathfrak{g}:=\mathfrak{so}(\Lambda_x^+)$, we use the invariant inner product
$$
(A,B)_{\mathfrak{g}}:=-\frac{1}{2}\operatorname{tr}(AB).
$$
The identification of $\Lambda_x^+$ with $\mathfrak{so}(\Lambda_x^+)$ which assigns $S_f$ to $f$ is an isometry for the fibre metric on $\Lambda_x^+$ and the invariant pairing $(-,-)_{\mathfrak{g}}$ on $\mathfrak{so}(\Lambda_x^+)$. This identification also preserves the Lie brackets. Consequently, every local oriented orthonormal frame $(\Sigma_1,\Sigma_2,\Sigma_3)$ of $\Lambda^+$ satisfies
$$
[\Sigma_i,\Sigma_j]=\sum_{k=1}^3\epsilon_{ijk}\Sigma_k
$$
for $i,j\in\{1,2,3\}$, where $\epsilon_{ijk}$ is the Levi--Civita symbol. We use the same cross-product identification and invariant inner product for every oriented Euclidean bundle $\mathcal{E}\rightarrow M$ of rank three.

Let $\nabla$ denote the Levi--Civita connection of $g$ on the tangent bundle $TM$. For smooth vector fields $X,Y,Z$ on $M$, our convention for the sign of the Riemannian curvature tensor is
$$
R(X,Y)Z=\nabla_X\nabla_YZ-\nabla_Y\nabla_XZ-\nabla_{[X,Y]}Z.
$$
For any point $x\in M$ and any orthonormal pair $X,Y\in T_xM$, the sectional curvature of the plane spanned by $X$ and $Y$ is $\langle R(X,Y)Y,X\rangle$. Thus, the round sphere has positive sectional curvature with the stated curvature convention.

All identifications between a Euclidean bundle and its dual are made using its fixed metric. In particular, a symmetric tensor on $\Lambda^+$ will always be regarded as a self-adjoint endomorphism $h\colon\Lambda^+\rightarrow\Lambda^+$. For a point $x\in M$ and vectors $v,w\in\Lambda_x^+$, the metric identification sends the tensor $v\otimes w$ to the endomorphism of $\Lambda_x^+$ whose value on a vector $z\in\Lambda_x^+$ is $\langle w,z\rangle v$. When $h$ is invertible, the symbol $h^{-1}$ denotes the inverse endomorphism of $\Lambda^+$. The condition $h>0$ means that $$\langle h(v),v\rangle>0$$ for every non-zero vector $v\in\Lambda_x^+$ and every $x\in M$. We denote by $\operatorname{Id}$ the identity endomorphism of $\Lambda^+$, and use the corresponding identity on a specified fibre when working at a point. The symbol $g$ will always denote the metric on $M$.

For any smooth endomorphism field $h$ of $\Lambda^+$, choose an open subset $U\subseteq M$ and a local oriented orthonormal frame $(\Sigma_1,\Sigma_2,\Sigma_3)$ of $\Lambda^+|_{U}$ and define the smooth real-valued functions $h_{11},h_{12},\dots,h_{33}$ on $U$ by $$h_{ij}:=\langle h(\Sigma_j),\Sigma_i\rangle.$$ Define the $\Lambda^+$-valued self-dual $2$-form $B_h$ by
\begin{equation}\label{form_from_endomorphism}
B_h|_{U}:=\sum_{i,j=1}^3h_{ij}\Sigma_i\otimes\Sigma_j.
\end{equation}
In each tensor $\Sigma_i\otimes\Sigma_j$ in equation~(\ref{form_from_endomorphism}), the factor $\Sigma_i$ is regarded as an ordinary self-dual $2$-form on $U$, and the factor $\Sigma_j$ is regarded as a section of the coefficient bundle $\Lambda^+|_{U}$. Thus, the tensor $B_h|_{U}$ in equation~(\ref{form_from_endomorphism}) is a section of $\Lambda^2\otimes\Lambda^+$ over $U$. Under a change of local oriented orthonormal frame of $\Lambda^+|_{U}$, the matrix entries $h_{ij}$ and both occurrences of the frame transform together. The resulting section $B_h|_{U}$ is unchanged, so equation~(\ref{form_from_endomorphism}) defines a global form $B_h\in\Omega^2(\Lambda^+)$.

The ordinary self-dual $2$-forms $B_h^1,B_h^2,B_h^3$ are defined to be the coefficients in the expression
$$
B_h|_{U}=\sum_{i=1}^3B_h^i\otimes\Sigma_i.
$$
For every $i\in\{1,2,3\}$, these coefficient forms $B_h^1,B_h^2,B_h^3$ are given by
$$
B_h^i=\sum_{j=1}^3h_{ji}\Sigma_j.
$$
If $h$ is self-adjoint, then $h_{ji}=h_{ij}$, so this expression for $B_h^i$ agrees with the expression in the introduction. Equivalently, equation~(\ref{form_from_endomorphism}) can be written as
$$
B_h|_{U}=\sum_{i=1}^3h(\Sigma_i)\otimes\Sigma_i.
$$
Thus, in the coefficient frame $(\Sigma_1,\Sigma_2,\Sigma_3)$ of $\Lambda^+|_{U}$, the coefficient of $B_h$ along $\Sigma_i$ is the ordinary self-dual $2$-form $h(\Sigma_i)$. The definition of $B_h$ is linear in $h$ and does not require positivity or self-adjointness.

For any self-adjoint endomorphism field $\eta\in\Omega^0(\operatorname{Sym}^2(\Lambda^+))$, we therefore write
$$
B_\eta|_{U}:=\sum_{i=1}^3\eta(\Sigma_i)\otimes\Sigma_i.
$$
In this expression for $B_\eta|_{U}$, the section $\eta(\Sigma_i)$ is regarded as an ordinary self-dual $2$-form on $U$, and the second factor $\Sigma_i$ is a section of the coefficient bundle $\Lambda^+|_{U}$. These local expressions define a global form $B_\eta\in\Omega^2(\Lambda^+)$.

More generally, let $\mathcal{E}\rightarrow M$ be an oriented Euclidean bundle of rank three. Choose an open subset $U\subseteq M$ and local oriented orthonormal frames $(e_1,e_2,e_3)$ of $\mathcal{E}|_{U}$ and $(\Sigma_1,\Sigma_2,\Sigma_3)$ of $\Lambda^+|_{U}$. An $\mathcal{E}$-valued self-dual $2$-form $T$ has a local expression
$$
T|_{U}=\sum_{i,j=1}^3T_{ji}\Sigma_j\otimes e_i,
$$
where $T_{ji}$ are smooth real-valued functions on $U$. Contracting the coefficient factor of $T$ using the metric of $\mathcal{E}$ gives the homomorphism from $\mathcal{E}|_{U}$ to $\Lambda^+|_{U}$ whose value on a section $v$ of $\mathcal{E}|_{U}$ is
\begin{equation}\label{curvature_homomorphism}
T(v):=\sum_{i,j=1}^3T_{ji}\langle e_i,v\rangle\Sigma_j.
\end{equation}
In equation~(\ref{curvature_homomorphism}), the pairing $\langle e_i,v\rangle$ is the fibre inner product on $\mathcal{E}$. We use equation~(\ref{curvature_homomorphism}) whenever an $\mathcal{E}$-valued self-dual $2$-form is regarded as a bundle homomorphism from $\mathcal{E}$ to $\Lambda^+$. In particular, the matrix of the homomorphism $T\colon\mathcal{E}\rightarrow\Lambda^+$ has row index $j$ in the ordinary-form frame $(\Sigma_1,\Sigma_2,\Sigma_3)$ and column index $i$ in the coefficient frame $(e_1,e_2,e_3)$. When $\mathcal{E}=\Lambda^+$, a product such as $F_A^+h^{-1}$ is a composition of endomorphisms of $\Lambda^+$, with $h^{-1}$ applied first. The equality $F_A^+=h$ is an equality of endomorphisms of $\Lambda^+$, and its equivalent equality of $\Lambda^+$-valued self-dual $2$-forms is $F_A^+=B_h$.

Let $A$ be an $\operatorname{SO}(3)$ connection on the oriented Euclidean bundle $\mathcal{E}\rightarrow M$, and denote by $\nabla^A$ the covariant derivative on sections of $\mathcal{E}$ determined by $A$. Choose a local oriented orthonormal frame $(e_1,e_2,e_3)$ of $\mathcal{E}|_{U}$. There are uniquely determined ordinary real-valued connection $1$-forms $A^i_j\in\Omega^1(U)$ such that, for every $j\in\{1,2,3\}$, we have
$$
\nabla^A e_j=\sum_{i=1}^3A^i_j\otimes e_i.
$$
Since $A$ preserves the metric of $\mathcal{E}$, the connection $1$-forms satisfy $A^i_j=-A^j_i$ for all $i,j\in\{1,2,3\}$.

For an integer $k\geq0$, every $\mathcal{E}$-valued $k$-form $\alpha$ has a unique local expression
$$
\alpha|_{U}=\sum_{i=1}^3\alpha_i\otimes e_i,
$$
where $\alpha_1,\alpha_2,\alpha_3\in\Omega^k(U)$ are ordinary real-valued differential $k$-forms on $U$. The exterior covariant derivative $d_A\colon\Omega^k(\mathcal{E})\rightarrow\Omega^{k+1}(\mathcal{E})$ is defined by
\begin{equation}\label{exterior_derivative}
d_A\alpha|_{U}=\sum_{i=1}^3\left(d\alpha_i+\sum_{j=1}^3A^i_j\wedge\alpha_j\right)\otimes e_i.
\end{equation}
In equation~(\ref{exterior_derivative}), the form $d\alpha_i$ is the ordinary exterior derivative of the coefficient form $\alpha_i$. The expression for $d_A\alpha|_{U}$ in equation~(\ref{exterior_derivative}) is unchanged under a change of local oriented orthonormal frame of $\mathcal{E}|_{U}$. These local expressions therefore define a global $\mathcal{E}$-valued $(k+1)$-form $d_A\alpha$.

In particular, for a connection $A$ on $\Lambda^+$, take the coefficient frame $(\Sigma_1,\Sigma_2,\Sigma_3)$ of $\Lambda^+|_{U}$ and let $A^i_j$ denote the connection $1$-forms in this frame. Applying equation~(\ref{exterior_derivative}) to $B_h$ then immediately gives
$$
d_AB_h|_{U}=\sum_{i=1}^3\left(dB_h^i+\sum_{j=1}^3A^i_j\wedge B_h^j\right)\otimes\Sigma_i.
$$
The ordinary exterior derivatives of the coefficient $2$-forms $B_h^i$ satisfy
$$
dB_h^i=\sum_{j=1}^3(dh_{ji}\wedge\Sigma_j+h_{ji}d\Sigma_j).
$$
Thus, the expression for $d_AB_h$ contains the derivatives of the functions $h_{ji}$, the ordinary exterior derivatives of the $2$-forms $\Sigma_j$, and the connection terms for the coefficient bundle $\Lambda^+$. Equivalently, the Levi--Civita connection of $g$ on $\Lambda^k$ and the connection $A$ on $\mathcal{E}$ define a tensor product connection on $\Lambda^k\otimes\mathcal{E}$. Alternating the covariant derivative for this tensor product connection gives the operator $d_A$ in equation~(\ref{exterior_derivative}). We use $\nabla^A$ for covariant differentiation of sections of $\mathcal{E}$ and $d_A$ for exterior covariant differentiation of $\mathcal{E}$-valued differential forms.

For an $\mathcal{E}$-valued $k$-form $\alpha$ and an $\mathcal{E}$-valued $l$-form $\beta$, write their coefficient forms in a local oriented orthonormal frame $(e_1,e_2,e_3)$ of $\mathcal{E}|_{U}$ as $\alpha_i\in\Omega^k(U)$ and $\beta_j\in\Omega^l(U)$. The bracket of $\alpha$ and $\beta$ is the $\mathcal{E}$-valued $(k+l)$-form defined by
$$
[\alpha,\beta]|_{U}:=\sum_{i,j=1}^3\alpha_i\wedge\beta_j\otimes[e_i,e_j].
$$
In the definition of $[\alpha,\beta]$, the expression $[e_i,e_j]$ is the cross product on the coefficient bundle $\mathcal{E}$, and $\alpha_i\wedge\beta_j$ is the ordinary exterior product of differential forms on $U$.

We denote by $F_A\in\Omega^2(\mathcal{E})$ the curvature of $A$, using the cross-product identification of $\mathfrak{so}(\mathcal{E})$ with $\mathcal{E}$. In a fixed local oriented orthonormal frame of $\mathcal{E}$, we also use $A$ for the local Lie-algebra-valued connection $1$-form. With this local convention, the curvature is given by
$$
F_A=dA+\frac{1}{2}[A,A],
$$
where $dA$ means the ordinary exterior derivative of the local coefficient $1$-forms of $A$. For every $\alpha\in\Omega^k(\mathcal{E})$, the curvature identity is $d_A^2\alpha=[F_A,\alpha]$. A local connection $1$-form depends on the chosen frame, whereas the difference of two connections on $\mathcal{E}$ is a global section of $\Lambda^1\otimes\mathfrak{so}(\mathcal{E})$. Using the cross-product identification, we regard that difference as an element of $\Omega^1(\mathcal{E})$.

For any $\mathcal{E}$-valued $2$-form $T$, the Hodge star acts on the ordinary $2$-form factor and leaves the coefficient factor in $\mathcal{E}$ unchanged. The self-dual and anti-self-dual parts of $T$ are defined by
$$
T^\pm:=\dfrac{T\pm\star T}{2}.
$$
A self-dual frame is an $\mathcal{E}$-valued self-dual $2$-form $T$ for which the homomorphism $T\colon\mathcal{E}\rightarrow\Lambda^+$ in equation~(\ref{curvature_homomorphism}) is an orientation-preserving isomorphism at every point of $M$. In particular, the form $B_h$ is a self-dual frame whenever $h$ is positive and self-adjoint.

Finally, for a positive self-adjoint endomorphism field $h$ of $\Lambda^+$, define $\operatorname{ad}_h(a):=[a,B_h]$ for every $a\in\Omega^1(\Lambda^+)$. Since the bracket is algebraic in $a$ at each point, the operator $\operatorname{ad}_h$ is induced by a bundle homomorphism from $\Lambda^1\otimes\Lambda^+$ to $\Lambda^3\otimes\Lambda^+$. Both of these bundles have rank $12$. The notation $\operatorname{ad}_h$ will be used only for this homomorphism and the corresponding operator on sections. For each point $x\in M$, we write $\operatorname{ad}_{h_x}:=(\operatorname{ad}_h)_x$ for the induced linear transform from $T_x^*M\otimes\Lambda_x^+$ to $\Lambda_x^3\otimes\Lambda_x^+$. For a vector $f\in\Lambda_x^+$, the notation $S_f$ will be used for the skew-adjoint endomorphism of $\Lambda_x^+$ defined by $S_f(v)=[f,v]$. For a section $f$ of $\Lambda^+$, the same notation denotes the resulting field of skew-adjoint endomorphisms of $\Lambda^+$.
\subsection{Basics of the frame formalism}

We first prove that, for a positive self-adjoint endomorphism $h$ of $\Lambda^+$, the equation $d_AB_h=0$ determines a unique $\operatorname{SO}(3)$ connection $A$ on $\Lambda^+$. The pointwise invertibility of $\operatorname{ad}_h$ will also be used in the calculation of the principal symbol of $D_h\Phi$.

\begin{definition}
For a smooth Euclidean bundle $\mathcal{E}\rightarrow M$, denote by
$$
\Omega^0_+(\operatorname{Sym}^2(\mathcal{E})):=\{h\in\Omega^0(\operatorname{Sym}^2(\mathcal{E})):h>0\}
$$
the space of smooth positive self-adjoint endomorphisms of $\mathcal{E}$.
\end{definition}

We use the fixed metric of $\mathcal{E}$ to identify sections of $\operatorname{Sym}^2(\mathcal{E})$ with self-adjoint endomorphism fields of $\mathcal{E}$. The condition $h>0$ means that each tensor $h_x$ defines a positive symmetric bilinear form on the dual vector space $\mathcal{E}_x^\vee$.

\begin{lemma}\label{Adjoint_action}
For any $h\in\Omega^0_+(\operatorname{Sym}^2(\Lambda^+))$, the adjoint action $\operatorname{ad}_h\colon\Omega^1(\Lambda^+)\rightarrow\Omega^3(\Lambda^+)$ is an isomorphism of vector spaces.
\end{lemma}
\begin{proof}[Proof of Lemma~\ref{Adjoint_action}]
Fix a point $x\in M$. Choose an oriented orthonormal basis $(e_1,e_2,e_3)$ of the vector space $\Lambda_x^+$ which diagonalizes the self-adjoint endomorphism $h_x$, and write
$$
h_x=\sum_{i=1}^3\lambda_i e_i\otimes e_i,
$$
where $\lambda_1,\lambda_2,\lambda_3$ are the positive eigenvalues of $h_x$. For each $i\in\{1,2,3\}$, define an endomorphism $J_i$ of $T_x^*M$ by $J_i\alpha:=-\sqrt{2}\star(\alpha\wedge e_i)$ for $\alpha\in T_x^*M$, where $\star$ is the Hodge star at $x$. We have that $J_i^2=-\operatorname{Id}_{T_x^*M}$, while $J_1J_2=J_3$ and $J_1J_3=-J_2$.

Take $a=a_1\otimes e_1+a_2\otimes e_2+a_3\otimes e_3\in T_x^*M\otimes\Lambda_x^+$, where $a_1,a_2,a_3\in T_x^*M$ are the coefficient covectors of $a$ in the basis $(e_1,e_2,e_3)$. The equation $[a,(B_h)_x]=0$ is equivalent to the linear system
$$
\begin{cases}
\lambda_3J_3a_2-\lambda_2J_2a_3=0,\\
\lambda_1J_1a_3-\lambda_3J_3a_1=0,\\
\lambda_2J_2a_1-\lambda_1J_1a_2=0.
\end{cases}
$$
The second and third equations of this linear system give that
$$
\lambda_1a_2=-\lambda_2J_3a_1
$$
and that
$$
\lambda_1a_3=\lambda_3J_2a_1.
$$
Substitution of these expressions for $a_2$ and $a_3$ into the first equation of the linear system then gives
$$
2\frac{\lambda_2\lambda_3}{\lambda_1}a_1=0.
$$
Thus, we obtain that $a_1=a_2=a_3=0$. Hence, the fibre homomorphism $(\operatorname{ad}_h)_x$ is injective. Both $T_x^*M\otimes\Lambda_x^+$ and $\Lambda_x^3\otimes\Lambda_x^+$ have dimension $12$, so $(\operatorname{ad}_h)_x$ is an isomorphism. Since $x$ was arbitrary, the smooth bundle homomorphism $\operatorname{ad}_h$ is a bundle isomorphism with smooth inverse, and therefore induces an isomorphism between the spaces of smooth sections $\Omega^1(\Lambda^+)$ and $\Omega^3(\Lambda^+)$. The diagonalization used in this argument is only at the fixed point $x$; no smoothly varying basis of eigenvectors is required.
\end{proof}

\begin{definition}
Let $\Gamma$ be the $\operatorname{SO}(3)$ connection on $\Lambda^+$ induced by the Levi--Civita connection of $g$. For any $h\in\Omega^0_+(\operatorname{Sym}^2(\Lambda^+))$, define the connection on $\Lambda^+$ by
$$
A(h):=\Gamma-\operatorname{ad}_h^{-1}(d_\Gamma B_h).
$$
\end{definition}

In the above definition of $A(h)$, the exterior covariant derivative $d_\Gamma$ acts on the $\Lambda^+$-valued $2$-form $B_h$, so $d_\Gamma B_h\in\Omega^3(\Lambda^+)$. By Lemma~\ref{Adjoint_action}, the form $\operatorname{ad}_h^{-1}(d_\Gamma B_h)$ belongs to $\Omega^1(\Lambda^+)$. Through the fixed cross-product identification of $\Lambda^+$ with $\mathfrak{so}(\Lambda^+)$, the $1$-form $\operatorname{ad}_h^{-1}(d_\Gamma B_h)$ takes values in the skew-adjoint endomorphisms of $\Lambda^+$ and may be subtracted from $\Gamma$. Thus, both $A(h)$ and $\Gamma$ are $\operatorname{SO}(3)$ connections on the fixed Euclidean bundle $\Lambda^+$.

\begin{lemma}\label{Self-Adjoint}
For any $h\in\Omega^0_+(\operatorname{Sym}^2(\Lambda^+))$, the endomorphism $F_{A(h)}^+h^{-1}$ of $\Lambda^+$ is self-adjoint.
\end{lemma}

In the composition $F_{A(h)}^+h^{-1}$ in Lemma~\ref{Self-Adjoint}, the self-dual curvature $F_{A(h)}^+$ is regarded as the endomorphism of $\Lambda^+$ obtained by contracting the bundle coefficient of the $\Lambda^+$-valued self-dual $2$-form $F_{A(h)}^+$ with the input vector, using the fixed metric of $\Lambda^+$.

\begin{proof}[Proof of Lemma~\ref{Self-Adjoint}]
Write $A:=A(h)$. The definition of $A(h)$ gives that $d_AB_h=0$. The Ricci identity for the exterior covariant derivative $d_A$ therefore gives
$$
[F_A,B_h]=d_A^2B_h=0.
$$
Since $B_h$ is self-dual, the wedge product of each coefficient form of $F_A^-$ with each coefficient form of $B_h$ vanishes. Thus, we have that $[F_A^-,B_h]=0$, and hence $[F_A^+,B_h]=0$.

Choose a local oriented orthonormal frame $(\Sigma_1,\Sigma_2,\Sigma_3)$ of $\Lambda^+$ over an open subset of $M$. Define the smooth coefficient functions $F_{ji}$ by
$$
F_A^+=\sum_{i,j=1}^3F_{ji}\Sigma_j\otimes\Sigma_i.
$$
Thus, the matrix $(F_{ji})$ represents the endomorphism $F_A^+$ of $\Lambda^+$ in the frame $(\Sigma_1,\Sigma_2,\Sigma_3)$, with the row index referring to the ordinary $2$-form and the column index referring to the bundle coefficient. Denote by $h_{ij}$ the matrix entries of $h$ in this frame. For $i,j\in\{1,2,3\}$, let $\delta_{ij}$ be the Kronecker symbol, equal to $1$ when $i=j$ and to $0$ otherwise. Using $\Sigma_i\wedge\Sigma_j=\delta_{ij}d\mu_g$, the identity $[F_A^+,B_h]=0$ says that the matrix whose $(i,j)$-entry is
$$
\sum_{k=1}^3F_{ki}h_{kj}
$$
is symmetric. In matrix notation, we obtain
$$
(F_A^+)^Th=hF_A^+,
$$
where the superscript $T$ denotes transpose in the orthonormal frame, equivalently the adjoint for the fixed metric of $\Lambda^+$. Consequently, we obtain
$$
(F_A^+h^{-1})^T=h^{-1}(F_A^+)^T=F_A^+h^{-1}
$$
as required.
\end{proof}

\begin{definition}
Define the operator
$\Phi:=\Phi_g\colon\Omega^0_+(\operatorname{Sym}^2(\Lambda^+))\rightarrow\Omega^0(\operatorname{Sym}^2(\Lambda^+))$ by
$$
\Phi(h):=F_{A(h)}^+h^{-1}-\operatorname{Id}.
$$
\end{definition}

In the definition of $\Phi(h)$, the curvature $F_{A(h)}^+$, the inverse $h^{-1}$, and the identity $\operatorname{Id}$ are all regarded as endomorphisms of the fixed bundle $\Lambda^+$. Lemma~\ref{Self-Adjoint} ensures that $\Phi(h)$ is self-adjoint. The operator $\Phi$ extends uniquely by continuity to the Sobolev spaces as
$$
\Phi\colon W^{s,p}_+(\operatorname{Sym}^2(\Lambda^+))\rightarrow W^{s-2,p}(\operatorname{Sym}^2(\Lambda^+)),
$$
where
$$
W^{s,p}_+(\operatorname{Sym}^2(\Lambda^+)):=\{h\in W^{s,p}(\operatorname{Sym}^2(\Lambda^+)):h>0\}.
$$
Since $s\geq2$ and $p>4$, every such Sobolev section has a continuous representative, and the condition $h>0$ means positivity at every point of $M$. The continuity and uniqueness of the extension of $\Phi$ follow from Lemma~\ref{Sobolev_smoothness} below.

For comparison with the triple notation, fix $h\in\Omega^0_+(\operatorname{Sym}^2(\Lambda^+))$, write $A:=A(h)$, and choose a local oriented orthonormal frame $(\Sigma_1,\Sigma_2,\Sigma_3)$ of $\Lambda^+$. Write
$$
F_A=\sum_{i=1}^3F_A^i\otimes\Sigma_i,
$$
where $F_A^1,F_A^2,F_A^3$ are ordinary $2$-forms, and recall that $B_h^1,B_h^2,B_h^3$ denote the ordinary self-dual coefficient $2$-forms of $B_h$ in the same frame. Define a matrix-valued function $X=(X_{ij})$ in this frame by
$$
X_{ij}d\mu_g:=F_A^i\wedge B_h^j.
$$
The proof of Lemma~\ref{Self-Adjoint} gives that $X=(F_A^+)^Th$ and that $X$ is symmetric. In the identity $X=(F_A^+)^Th$, the endomorphisms $F_A^+$ and $h$ of $\Lambda^+$ are represented in the frame $(\Sigma_1,\Sigma_2,\Sigma_3)$. The identity $(F_A^+)^Th=hF_A^+$ and the definition of $\Phi(h)$ give
$$
X=h(\operatorname{Id}+\Phi(h))h.
$$
Thus, the equation $\Phi(h)=0$ is equivalent to $X=h^2$, or equivalently to $h^{-1}Xh^{-1}=\operatorname{Id}$. Since $B_h^i\wedge B_h^j=(h^2)_{ij}d\mu_g$, the matrix $X$ equals $h^2$ at every solution $h$ of $\Phi(h)=0$, and multiplication by $h^{-1}$ on both sides is required to obtain the identity matrix. In terms of ordinary $2$-forms, the equation $\Phi(h)=0$ is equivalently
$$
F_A^i\wedge\sum_{k=1}^3(h^{-1})_{kj}\Sigma_k=\delta_{ij}d\mu_g
$$
for every $i,j\in\{1,2,3\}$, where $(h^{-1})_{ij}$ denotes the $(i,j)$-entry of the inverse endomorphism $h^{-1}$ of $\Lambda^+$ in the frame $(\Sigma_1,\Sigma_2,\Sigma_3)$.

We next give a variational formulation of the equation $\Phi(h)=0$. For this purpose, we specify the inner products used in the functional $E[h]$, which will be defined in the next proposition. For endomorphisms $X,Y$ of $\Lambda^+$, we use the pointwise Hilbert--Schmidt inner product
$$
\langle X,Y\rangle:=\operatorname{tr}(X^TY)
$$
and the squared norm $|X|^2:=\langle X,X\rangle$. The Hilbert--Schmidt inner product on endomorphisms is distinct from the normalized inner product $(-,-)_{\mathfrak{g}}$ used for Lie algebra coefficients. For $\Lambda^+$-valued differential forms, the pointwise inner product is induced by $g$ on the ordinary form factor and by the fixed Euclidean metric on the coefficient bundle $\Lambda^+$. These conventions give $\langle B_X,B_Y\rangle=\langle X,Y\rangle$ for endomorphisms $X,Y$ of $\Lambda^+$. All $L^2$ inner products below are obtained by integrating the corresponding pointwise inner products against the volume form $d\mu_g$ of $(M,g)$.

\begin{proposition}\label{E-L}
The Euler--Lagrange equation of the functional
\begin{equation}\label{Functional}
E[h]:=\frac{1}{2}\int_M(2\langle F_{A(h)}^+,h\rangle-|h|^2)d\mu_g
\end{equation}
on $\Omega^0_+(\operatorname{Sym}^2(\Lambda^+))$ is precisely $\Phi(h)=0$.
\end{proposition}

In equation~(\ref{Functional}), the self-dual curvature $F_{A(h)}^+$ is regarded as an endomorphism of $\Lambda^+$, and the pairing with $h$ is the Hilbert--Schmidt inner product on endomorphisms of $\Lambda^+$.

\begin{proof}[Proof of Proposition~\ref{E-L}]
Fix a positive self-adjoint endomorphism $h\in\Omega^0_+(\operatorname{Sym}^2(\Lambda^+))$.

Consider a smooth family $h(t)\in\Omega^0_+(\operatorname{Sym}^2(\Lambda^+))$ for $t$ in an open interval containing $0$, with $h(0)=h$, and write $\eta:=\dot h(0)\in\Omega^0(\operatorname{Sym}^2(\Lambda^+))$. Define $A(t):=A(h(t))$, $A:=A(0)$, and $a:=\dot A(0)\in\Omega^1(\Lambda^+)$. The curvature variation formula gives
$$
\left.\frac{d}{dt}\right|_{t=0}F_{A(t)}=d_Aa.
$$
The compatibility identity $d_{A(t)}B_{h(t)}=0$ at $t=0$ gives $d_AB_h=0$. Denote by $d_A^*$ the formal $L^2$ adjoint of $d_A$ for the fixed metrics on $M$ and $\Lambda^+$. On $\Lambda^+$-valued $2$-forms, the adjoint is $d_A^*=-\star d_A\star$, with the Hodge star acting on the ordinary form factor. Since $\star B_h=B_h$, we obtain that $d_A^*B_h=0$. As $M$ is closed, integration by parts gives
$$
\int_M\langle d_Aa,B_h\rangle d\mu_g =\int_M\langle a,d_A^*B_h\rangle d\mu_g=0.
$$
Here, the pairing $\langle d_Aa,B_h\rangle$ is the pointwise inner product of $\Lambda^+$-valued $2$-forms. Differentiating equation~(\ref{Functional}) therefore gives
$$
\left.\frac{d}{dt}\right|_{t=0}E[h(t)] =\int_M(\langle d_Aa,B_h\rangle+\langle F_A^+-h,\eta\rangle)d\mu_g,
$$
and hence
$$
\left.\frac{d}{dt}\right|_{t=0}E[h(t)] =\int_M\langle F_A^+-h,\eta\rangle d\mu_g.
$$
In the pairing with $\eta$, the curvature $F_A^+$ is regarded as an endomorphism of $\Lambda^+$. Therefore, every solution to $\Phi(h)=0$ is a critical point of $E$.

Conversely, suppose that $h$ is a critical point of $E$. Every smooth self-adjoint endomorphism field $\eta$ of $\Lambda^+$ is an admissible variation, since $h+t\eta$ remains positive for sufficiently small $|t|$ by compactness of $M$. The first variation identity for $E[h(t)]$ therefore implies that the endomorphism $T:=F_A^+-h$ of $\Lambda^+$ is skew-adjoint at every point of $M$. By Lemma~\ref{Self-Adjoint}, the endomorphism $Th^{-1}=F_A^+h^{-1}-\operatorname{Id}$ of $\Lambda^+$ is self-adjoint. Taking its adjoint and using $T^T=-T$ gives
$$
h^{-1}T+Th^{-1}=0.
$$
To conclude, fix $x\in M$ and choose an orthonormal basis of $\Lambda_x^+$ which diagonalizes $h_x^{-1}$, with positive eigenvalues $\mu_1,\mu_2,\mu_3$. If $T_{ij}$ are the matrix entries of $T_x$ in this basis, the identity $h^{-1}T+Th^{-1}=0$ gives $(\mu_i+\mu_j)T_{ij}=0$ for every $i,j\in\{1,2,3\}$. Thus, we obtain that $T_x=0$ at every point $x$. Consequently, $F_A^+=h$, and hence $\Phi(h)=0$.
\end{proof}

More precisely, the $L^2$ gradient of $E$ on $\Omega^0_+(\operatorname{Sym}^2(\Lambda^+))$ is the self-adjoint part of the endomorphism $F_{A(h)}^+-h$ of $\Lambda^+$. We write $\operatorname{grad}(E(h))$ for this gradient evaluated at $h$. Since $F_{A(h)}^+-h=\Phi(h)h$, we obtain
$$
\operatorname{grad}(E(h))=\frac{\Phi(h)h+h\Phi(h)}{2}.
$$
For each positive $h_x$, consider the linear map on self-adjoint endomorphisms of $\Lambda_x^+$ which sends $Q$ to $(Qh_x+h_xQ)/2$. Choose an orthonormal eigenbasis of $h_x$, and denote its positive eigenvalues by $\lambda_1,\lambda_2,\lambda_3$. In this basis, the linear map multiplies the $(i,j)$-entry of $Q$ by the positive number $(\lambda_i+\lambda_j)/2$. Hence, the linear map is invertible. Thus, the equation $\operatorname{grad}(E(h))=0$ is equivalent to $\Phi(h)=0$, although the two operators are generally different.

The functional $E$ also has a direct first order interpretation. Let $A$ be an $\operatorname{SO}(3)$ connection on the fixed Euclidean bundle $\Lambda^+$, and let $B\in\Omega^2(\Lambda^+)$ be an unrestricted bundle-valued $2$-form, independent of $A$. Choose a local oriented orthonormal frame $(e_1,e_2,e_3)$ of the coefficient bundle $\Lambda^+$. We write
$$
B=\sum_{i=1}^3B^i\otimes e_i,
$$
and also write
$$
F_A=\sum_{i=1}^3F_A^i\otimes e_i,
$$
where $B^i$ and $F_A^i$ are ordinary $2$-forms on the domain of the chosen frame. Consider the action
$$
S[A,B]:=\int_M\sum_{i=1}^3\left(B^i\wedge F_A^i-\frac{1}{2}B^i\wedge B^i\right).
$$
In the definition of $S[A,B]$, contraction of the coefficient indices uses the fixed Euclidean metric of $\Lambda^+$. Consequently, the $4$-form under the integral is independent of the chosen local orthonormal frame and is defined globally on $M$. The functional $S$ is the BF action with a quadratic cosmological term \cite{freidel}; the letters BF refer to the pairing of the $2$-form variable $B$ with the curvature $F_A$. Varying $A$ and $B$ independently gives the Euler--Lagrange equations
$$
\begin{cases}
d_AB=0,\\
F_A=B.
\end{cases}
$$
Both equations concern bundle-valued forms: the first is an equality in $\Omega^3(\Lambda^+)$ and the second in $\Omega^2(\Lambda^+)$. Substituting $B=F_A$ into $S[A,B]$ gives the characteristic integral
$$
\frac{1}{2}\int_M\sum_{i=1}^3F_A^i\wedge F_A^i.
$$

If the variable $B$ is constrained to be self-dual for the fixed metric $g$ before taking variations, variation of $B$ gives $F_A^+=B$, while variation of $A$ still gives $d_AB=0$. Eliminating $B$ from these constrained equations gives $d_AF_A^+=0$. By the Bianchi identity $d_AF_A=0$ and the decomposition $F_A=F_A^++F_A^-$, the equation $d_AF_A^+=0$ is equivalent to $d_A^*F_A=0$, the Yang--Mills equation. The constrained equations therefore permit a non-zero anti-self-dual curvature tensor $F_A^-$.

On the subset of self-dual forms $B$ for which the associated endomorphism of $\Lambda^+$ is an orientation-preserving isomorphism, pointwise polar decomposition gives a positive self-adjoint endomorphism $h$ of $\Lambda^+$ after applying an orientation-preserving isometry to the coefficient bundle $\Lambda^+$. After this action on the coefficient bundle, the transformed $2$-form is $B_h$. The equation $d_AB_h=0$ determines $A=A(h)$, as proved in Lemma~\ref{Compatible_connection} below. Since $E[h]=S[A(h),B_h]$, Proposition~\ref{E-L} proves that restricting the BF action in this way gives precisely the equation $\Phi(h)=0$.

\begin{corollary}\label{Energy_identity}
If $h$ is a solution of $\Phi(h)=0$, then
$$
E[h]=\frac{1}{2}\int_M|h|^2d\mu_g.
$$
\end{corollary}
\begin{proof}[Proof of Corollary~\ref{Energy_identity}]
Since $\Phi(h)=0$, we have that $F_{A(h)}^+=h$ as endomorphisms of $\Lambda^+$. Substitution into equation~(\ref{Functional}) gives the stated formula for $E[h]$.
\end{proof}

\begin{lemma}\label{Compatible_connection}
For any $h\in\Omega^0_+(\operatorname{Sym}^2(\Lambda^+))$, the connection $A(h)$ is the unique $\operatorname{SO}(3)$ connection on the fixed Euclidean bundle $\Lambda^+$ satisfying $d_{A(h)}B_h=0$.
\end{lemma}
\begin{proof}[Proof of Lemma~\ref{Compatible_connection}]
By the definition of $A(h)$ and the identity $\operatorname{ad}_h(a)=[a,B_h]$ for $a\in\Omega^1(\Lambda^+)$, we have
$$
d_{A(h)}B_h=d_\Gamma B_h-[\operatorname{ad}_h^{-1}(d_\Gamma B_h),B_h]=0.
$$
Suppose that $A$ is another $\operatorname{SO}(3)$ connection on $\Lambda^+$ satisfying $d_AB_h=0$. Write $A=A(h)+a$, where $a\in\Omega^1(\Lambda^+)$ is the difference of the two connections under the cross-product identification. Then
$$
0=d_AB_h=d_{A(h)}B_h+[a,B_h]=[a,B_h].
$$
Lemma~\ref{Adjoint_action} gives that $a=0$, so $A=A(h)$. This proves the uniqueness of the connection compatible with $B_h$.
\end{proof}

\begin{remark}
The definition $A(h)=\Gamma-\operatorname{ad}_h^{-1}(d_\Gamma B_h)$ uses the reference connection $\Gamma$. However, Lemma~\ref{Compatible_connection} characterizes $A(h)$ uniquely by $d_{A(h)}B_h=0$, without reference to $\Gamma$. Equivalently, for any reference $\operatorname{SO}(3)$ connection $A_0$ on the fixed Euclidean bundle $\Lambda^+$, we have
$$
A(h)=A_0-\operatorname{ad}_h^{-1}(d_{A_0}B_h).
$$
Indeed, the connection $A_0-\operatorname{ad}_h^{-1}(d_{A_0}B_h)$ satisfies the same compatibility equation $d_AB_h=0$. Thus, the expression $A_0-\operatorname{ad}_h^{-1}(d_{A_0}B_h)$ is independent of the chosen reference connection $A_0$.
\end{remark}

We now consider the dependence of $A(h)$ and $\Phi(h)$ on $h$ in Sobolev spaces. With $\Gamma$ as the fixed reference connection, we write $(A-\Gamma)(h):=A(h)-\Gamma$.

\begin{lemma}\label{Sobolev_smoothness}
The subset $W^{s,p}_+(\operatorname{Sym}^2(\Lambda^+))$ is an open convex subset of $W^{s,p}(\operatorname{Sym}^2(\Lambda^+))$. Moreover, the mappings
$$
A-\Gamma\colon W^{s,p}_+(\operatorname{Sym}^2(\Lambda^+))\rightarrow W^{s-1,p}(\Lambda^1\otimes\Lambda^+)
$$
and
$$
\Phi\colon W^{s,p}_+(\operatorname{Sym}^2(\Lambda^+))\rightarrow W^{s-2,p}(\operatorname{Sym}^2(\Lambda^+))
$$
are smooth.
\end{lemma}
\begin{proof}[Proof of Lemma~\ref{Sobolev_smoothness}]
Since $s\geq2$ and $p>4$, for any smooth vector bundle $\mathcal{E}\rightarrow M$, the Sobolev embedding theorem gives a compact inclusion of $W^{s,p}(\mathcal{E})$ into the space of continuous sections of $\mathcal{E}$ on the closed $4$-manifold $M$.

For a continuous positive self-adjoint endomorphism field $h$ of $\Lambda^+$, compactness of $M$ gives a positive lower bound for the smallest eigenvalue of $h_x$ over all $x\in M$. Positivity is therefore open in the uniform topology and is preserved under convex combinations of positive self-adjoint fields. It follows that $W^{s,p}_+(\operatorname{Sym}^2(\Lambda^+))$ is open and convex in $W^{s,p}(\operatorname{Sym}^2(\Lambda^+))$. Smooth self-adjoint endomorphism fields of $\Lambda^+$ are dense in $W^{s,p}(\operatorname{Sym}^2(\Lambda^+))$, and convergence in this Sobolev space implies uniform convergence. Consequently, a sufficiently close smooth approximation of a positive section remains positive, proving that $\Omega^0_+(\operatorname{Sym}^2(\Lambda^+))$ is dense in $W^{s,p}_+(\operatorname{Sym}^2(\Lambda^+))$.

The pointwise inverse $h^{-1}$ depends smoothly on $h$ in the open subset $W^{s,p}_+(\operatorname{Sym}^2(\Lambda^+))$. Similarly, Lemma~\ref{Adjoint_action} and smooth dependence of matrix inversion show that $\operatorname{ad}_h^{-1}$ depends smoothly on $h$ as a section of the bundle of homomorphisms from $\Lambda^3\otimes\Lambda^+$ to $\Lambda^1\otimes\Lambda^+$. The formula $A(h)-\Gamma=-\operatorname{ad}_h^{-1}(d_\Gamma B_h)$ contains one derivative of $h$, and the formula $\Phi(h)=F_{A(h)}^+h^{-1}-\operatorname{Id}$ contains at most two derivatives of $h$. The Sobolev multiplication theorem therefore gives the asserted smoothness into $W^{s-1,p}(\Lambda^1\otimes\Lambda^+)$ and $W^{s-2,p}(\operatorname{Sym}^2(\Lambda^+))$, respectively.
\end{proof}

\subsection{The Yang--Mills interpretation}

We next identify the solutions of $\Phi(h)=0$ with Yang--Mills connections whose self-dual curvatures are orientation-preserving frames. Throughout this subsection, the Yang--Mills equation is $d_A\star F_A=0$, where $\star$ is the Hodge star of the fixed Riemannian metric $g$ and acts on the ordinary differential-form factor of $F_A$.

We first specify the transport of a connection between coefficient bundles. Let $A$ be an $\operatorname{SO}(3)$ connection on an oriented Euclidean bundle $\mathcal{E}\rightarrow M$ of rank three, and let $R\colon\mathcal{E}\rightarrow\Lambda^+$ be an orientation-preserving bundle isometry covering the identity of $M$. Denote by $R_*A$ the connection on $\Lambda^+$ defined by
$$
\nabla^{R_*A}_Xv:=R\nabla^A_X(R^{-1}v)
$$
for every smooth vector field $X$ on $M$ and every smooth section $v$ of $\Lambda^+$. The isometry $R$ acts on the coefficient bundle of the curvature, while the ordinary $2$-forms and their self-dual decomposition are determined by $g$. Consequently, when self-dual curvatures are regarded as homomorphisms as in equation~(\ref{curvature_homomorphism}), the curvature of the transported connection satisfies
$$
F_{R_*A}^+=F_A^+R^{-1}\colon\Lambda^+\rightarrow\Lambda^+.
$$

In particular, when $\mathcal{E}=\Lambda^+$, the isometry $R$ is a gauge transformation of $\Lambda^+$, namely an orientation-preserving isometric automorphism of $\Lambda^+$ covering the identity of $M$.

\begin{theorem}\label{Yang_Mills_equivalence}
For any $h\in\Omega^0_+(\operatorname{Sym}^2(\Lambda^+))$ satisfying $\Phi(h)=0$, the connection $A(h)$ on $\Lambda^+$ is Yang--Mills and its self-dual curvature, regarded as an endomorphism of $\Lambda^+$, is precisely $h$ itself.
Conversely, let $A$ be a smooth $\operatorname{SO}(3)$ Yang--Mills connection on an oriented Euclidean bundle $\mathcal{E}\rightarrow M$ of rank three. Suppose that the curvature homomorphism $F_A^+\colon\mathcal{E}\rightarrow\Lambda^+$ is an orientation-preserving isomorphism at every point of $M$. Then there is a unique orientation-preserving bundle isometry $R\colon\mathcal{E}\rightarrow\Lambda^+$ such that the self-dual curvature of $R_*A$, regarded as an endomorphism of $\Lambda^+$, is a positive self-adjoint endomorphism $h$. Moreover, it holds that $R_*A=A(h)$ and $\Phi(h)=0$. If $\mathcal{E}=\Lambda^+$, then $R$ is a gauge transformation.
\end{theorem}
\begin{proof}[Proof of Theorem~\ref{Yang_Mills_equivalence}]
Assume first that $\Phi(h)=0$, and write $A:=A(h)$. The definition of $\Phi$ gives $F_A^+=h$ as endomorphisms of $\Lambda^+$. By the correspondence in equation~(\ref{form_from_endomorphism}), the endomorphism identity $F_A^+=h$ is equivalent to the identity $F_A^+=B_h$ in $\Omega^2(\Lambda^+)$. Lemma~\ref{Compatible_connection} gives $d_AB_h=0$. The Bianchi identity $d_AF_A=0$ therefore gives $d_AF_A^-=0$. Consequently, we obtain
$$
d_A\star F_A=d_A(B_h-F_A^-)=0.
$$
Thus, the connection $A$ is Yang--Mills.

Conversely, write $F:=F_A^+$ for the curvature homomorphism from $\mathcal{E}$ to $\Lambda^+$. Let $F^*\colon\Lambda^+\rightarrow\mathcal{E}$ be the adjoint of $F$ with respect to the fixed Euclidean metrics on $\mathcal{E}$ and $\Lambda^+$, and put
$$
\begin{cases}
h:=(FF^*)^{1/2},\\
R:=h^{-1}F.
\end{cases}
$$
The endomorphism $FF^*$ of $\Lambda^+$ is positive and self-adjoint, so $h$ is the unique positive self-adjoint square root of $FF^*$. The map $R\colon\mathcal{E}\rightarrow\Lambda^+$ satisfies $RR^*=\operatorname{Id}_{\Lambda^+}$, where $R^*\colon\Lambda^+\rightarrow\mathcal{E}$ is its adjoint for the fixed Euclidean metrics, and is therefore a bundle isometry. Since $F$ is orientation-preserving and $h$ is positive, the isometry $R$ is orientation-preserving. The endomorphism field $h$ and the isometry $R$ are smooth and globally defined because the positive square root depends smoothly on a positive self-adjoint endomorphism of $\Lambda^+$.

The self-dual curvature homomorphism of the transported connection $R_*A$ is $FR^{-1}=h$. Equivalently, the self-dual curvature as a $\Lambda^+$-valued $2$-form is $F_{R_*A}^+=B_h$. The connection $R_*A$ is Yang--Mills because $R$ is a bundle isometry. The Bianchi identity and the Yang--Mills equation for $R_*A$ now give
$$
d_{R_*A}B_h=\frac{1}{2}d_{R_*A}(F_{R_*A}+\star F_{R_*A})=0.
$$
Lemma~\ref{Compatible_connection} therefore gives $R_*A=A(h)$. Since the self-dual curvature homomorphism of $R_*A$ is $h$, we conclude that $\Phi(h)=0$.

To prove uniqueness, write $R_1:=R$ for the isometry constructed above, and let $R_2\colon\mathcal{E}\rightarrow\Lambda^+$ be another orientation-preserving bundle isometry such that the endomorphism $h_2:=FR_2^{-1}$ of $\Lambda^+$ is positive and self-adjoint. The equality $F=h_2R_2$ and the isometry identity $R_2^*=R_2^{-1}$ give $FF^*=h_2^2$. Uniqueness of the positive self-adjoint square root of $FF^*$ implies $h_2=h$. It follows that $R_2=h^{-1}F=R_1$.
\end{proof}

The orientation hypothesis in the converse direction of Theorem~\ref{Yang_Mills_equivalence} is essential. Composition with an orientation-preserving bundle isometry cannot change the sign of the determinant of $F_A^+$. Thus, an orientation-reversing curvature homomorphism cannot become a positive self-adjoint endomorphism of $\Lambda^+$ under an $\operatorname{SO}(3)$ gauge transformation. Yang--Mills connections with orientation-reversing self-dual curvature homomorphisms therefore do not correspond to solutions of $\Phi(h)=0$ in the positive cone $\Omega^0_+(\operatorname{Sym}^2(\Lambda^+))$. When the coefficient bundle has been identified with $\Lambda^+$, symmetry of the curvature matrix means self-adjointness of the resulting endomorphism $F_A^+\colon\Lambda^+\rightarrow\Lambda^+$. A gauge transformation acts on the coefficient bundle while leaving the ordinary differential-form factor fixed, and therefore need not preserve this self-adjointness condition.

We next examine the irreducibility and stabilizer of $A(h)$ for a solution $h$ of $\Phi(h)=0$. The stabilizer of a connection $A(h)$ in the $\operatorname{SO}(3)$ gauge group is the group of gauge transformations $R$ of $\Lambda^+$ satisfying $R_*A(h)=A(h)$.

\begin{corollary}\label{Irreducibility}
Let $h\in\Omega^0_+(\operatorname{Sym}^2(\Lambda^+))$ be a classical solution to $\Phi(h)=0$. Then the $\operatorname{SO}(3)$ connection $A(h)$ on $\Lambda^+$ is irreducible. More precisely, if a smooth section $f\in\Omega^0(\Lambda^+)$ satisfies $d_{A(h)}f=0$, then $f=0$. Moreover, the stabilizer of $A(h)$ in the $\operatorname{SO}(3)$ gauge group is trivial.
\end{corollary}
\begin{proof}[Proof of Corollary~\ref{Irreducibility}]
Write $A:=A(h)$. If $d_Af=0$, then the Ricci identity gives
$$
[F_A,f]=d_A^2f=0.
$$
Taking the self-dual part in the ordinary $2$-form factor gives $[F_A^+,f]=0$. Since $F_A^+=B_h$, we obtain $[B_h,f]=0$. Fix a point $x\in M$. Since $h_x$ is invertible, the coefficient vectors of $(B_h)_x$, obtained by contracting its ordinary $2$-form factor against elements of $\Lambda_x^+$, span the coefficient space $\Lambda_x^+$. Hence, $[v,f_x]=0$ for every $v\in\Lambda_x^+$. The center of $\mathfrak{so}(3)$ is trivial, so the cross-product identification of $\Lambda_x^+$ with $\mathfrak{so}(\Lambda_x^+)$ gives $f_x=0$. Since $x$ was arbitrary, we conclude that $f=0$.

Now suppose that a gauge transformation $R$ satisfies $R_*A=A$. The equality $R_*A=A$ implies that $R$ is parallel for the connection induced by $A$ on $\operatorname{End}(\Lambda^+)$. The transformation rule for curvature also implies that the action of $R$ on the coefficient bundle fixes $F_A^+=B_h$. At each point $x\in M$, the coefficient vectors of $(B_h)_x$ span $\Lambda_x^+$, so $R_x$ fixes every vector of $\Lambda_x^+$. Thus, we obtain $R_x=\operatorname{Id}$ for every $x\in M$, and $R$ is the identity gauge transformation.
\end{proof}

\begin{remark}
The Yang--Mills moduli space considered here is the set of smooth Yang--Mills connections on $\Lambda^+$ modulo $\operatorname{SO}(3)$ gauge transformations, with the quotient topology induced by the $C^\infty$ topology on smooth connections. Theorem~\ref{Yang_Mills_equivalence} identifies the set $\Phi_g^{-1}(0)$ of smooth positive self-adjoint endomorphism fields $h$ of $\Lambda^+$ satisfying $\Phi_g(h)=0$ with the open subset of the Yang--Mills moduli space represented by connections whose self-dual curvature homomorphisms are orientation-preserving isomorphisms. Indeed, the construction $h=(F_A^+(F_A^+)^*)^{1/2}$ is unchanged by a gauge transformation of $A$, and Theorem~\ref{Yang_Mills_equivalence} gives a unique representative $A(h)$ of each such gauge class whose self-dual curvature is the positive self-adjoint endomorphism $h$ of $\Lambda^+$. Corollary~\ref{Irreducibility} shows that the stabilizer of every connection in this subset is trivial.
\end{remark}

\subsection{Linearization and ellipticity}

We now compute the derivative of $\Phi$ with respect to the positive self-adjoint endomorphism field $h$ of $\Lambda^+$. Both the variation of $h$ and the value of $\Phi(h)$ are sections of the rank-six bundle $\operatorname{Sym}^2(\Lambda^+)$. We shall prove that the principal symbol of $D_h\Phi$ is invertible for every non-zero cotangent vector. Thus, the equation $\Phi(h)=0$ is an elliptic system with the same number of equations and unknowns, and requires no additional gauge-fixing equation.

\begin{lemma}\label{Linearisation_formula}
Let $h(t)$, for $t$ in an open interval containing $0$, be a smooth variation in $\Omega^0_+(\operatorname{Sym}^2(\Lambda^+))$ with $h(0)=h$. Write $\eta:=\dot h(0)\in\Omega^0(\operatorname{Sym}^2(\Lambda^+))$. Then it holds that
$$
\left.\frac{d}{dt}\right|_{t=0}A(h(t)) =-\operatorname{ad}_h^{-1}(d_{A(h)}B_\eta).
$$
Moreover, the Fr\'echet derivative of $\Phi$ at $h$, applied to the variation $\eta$, is
\begin{equation}\label{Linearisation_Phi}
D_h\Phi(\eta) =-(d_{A(h)}\operatorname{ad}_h^{-1}(d_{A(h)}B_\eta))^+h^{-1} -F_{A(h)}^+h^{-1}\eta h^{-1}.
\end{equation}
Furthermore, if $h$ is a solution to $\Phi(h)=0$, then equation~(\ref{Linearisation_Phi}) reduces to
$$
D_h\Phi(\eta) =-(d_{A(h)}\operatorname{ad}_h^{-1}(d_{A(h)}B_\eta))^+h^{-1}-\eta h^{-1}.
$$
\end{lemma}

For the variation $h(t)$ in Lemma~\ref{Linearisation_formula}, the form $B_\eta$ is the derivative of the family of $\Lambda^+$-valued $2$-forms $B_{h(t)}$ at $t=0$:
$$
B_\eta=\left.\frac{d}{dt}\right|_{t=0}B_{h(t)}.
$$
Indeed, differentiating each coefficient $h(t)(\Sigma_i)$ of $B_{h(t)}$ in a fixed local oriented orthonormal frame $(\Sigma_1,\Sigma_2,\Sigma_3)$ of $\Lambda^+$ gives the coefficient $\eta(\Sigma_i)$ of $B_\eta$.

In equation~(\ref{Linearisation_Phi}), the exterior derivative $d_{A(h)}B_\eta$ belongs to $\Omega^3(\Lambda^+)$. Applying $\operatorname{ad}_h^{-1}$ gives an element of $\Omega^1(\Lambda^+)$, and applying $d_{A(h)}$ again gives an element of $\Omega^2(\Lambda^+)$. The superscript $+$ projects the ordinary $2$-form factor to its self-dual part. The resulting self-dual form and the curvature $F_{A(h)}^+$ are then regarded as endomorphisms of $\Lambda^+$, so every product with $h^{-1}$ or $\eta$ in equation~(\ref{Linearisation_Phi}) is a composition of endomorphisms of $\Lambda^+$.

\begin{proof}[Proof of Lemma~\ref{Linearisation_formula}]
Write $A(t):=A(h(t))$ and $A:=A(0)$. The derivative $\dot A(0)$ is a $\Lambda^+$-valued $1$-form, since $A(t)$ is a family of connections on the fixed bundle $\Lambda^+$. Differentiating the identity $d_{A(t)}B_{h(t)}=0$ at $t=0$, and using the linear dependence of $B_{h(t)}$ on $h(t)$, we obtain
$$
d_AB_\eta+[\dot A(0),B_h]=0.
$$
Since $[\dot A(0),B_h]=\operatorname{ad}_h(\dot A(0))$, Lemma~\ref{Adjoint_action} gives
$$
\dot A(0)=-\operatorname{ad}_h^{-1}(d_AB_\eta).
$$
The curvature variation formula is
$$
\left.\frac{d}{dt}\right|_{t=0}F_{A(t)}=d_A\dot A(0).
$$
Differentiating $h(t)h(t)^{-1}=\operatorname{Id}$ gives
$$
\left.\frac{d}{dt}\right|_{t=0}h(t)^{-1}=-h^{-1}\eta h^{-1}.
$$
We now differentiate $\Phi(h(t))=F_{A(t)}^+h(t)^{-1}-\operatorname{Id}$. The Hodge star is independent of $t$, because the metric $g$ is fixed. Substituting the curvature variation and the derivative of $h(t)^{-1}$ therefore gives equation~(\ref{Linearisation_Phi}). If $\Phi(h)=0$, then $F_A^+=h$ as endomorphisms of $\Lambda^+$, so $F_A^+h^{-1}\eta h^{-1}=\eta h^{-1}$. This proves the stated simplification of $D_h\Phi(\eta)$ at a solution $h$ of $\Phi(h)=0$.
\end{proof}

To compute the principal symbol of $D_h\Phi$, we first specify the meaning of $B_\eta$ when $\eta$ is an endomorphism of a single fibre of $\Lambda^+$. For a point $x\in M$ and a self-adjoint endomorphism $\eta\colon\Lambda_x^+\rightarrow\Lambda_x^+$, choose an oriented orthonormal basis $(\Sigma_1,\Sigma_2,\Sigma_3)$ of $\Lambda_x^+$ and define
$$
B_\eta:=\sum_{i=1}^3\eta(\Sigma_i)\otimes\Sigma_i\in\Lambda_x^+\otimes\Lambda_x^+.
$$
In this definition of $B_\eta$, the first factor $\eta(\Sigma_i)$ is an ordinary self-dual $2$-covector at $x$, and the second factor $\Sigma_i$ lies in the coefficient fibre $\Lambda_x^+$. The tensor $B_\eta$ is independent of the chosen oriented orthonormal basis of $\Lambda_x^+$. If a smooth self-adjoint endomorphism field $h\in\Omega^0(\operatorname{Sym}^2(\Lambda^+))$ satisfies $h_x=\eta$, then $(B_h)_x=B_\eta$.

For the principal-symbol calculation in the next Proposition~\ref{Principal_symbol}, we write $\eta=\eta_0+\eta_1$ for the decomposition of a self-adjoint endomorphism $\eta$ of $\Lambda_x^+$ into the trace-free part $\eta_0$ and the scalar part $\eta_1$, characterized by $\operatorname{tr}(\eta_0)=0$ and $3\eta_1=\operatorname{tr}(\eta)\operatorname{Id}$. For a cotangent vector $\xi\in T_x^*M$, the notation $|\xi|$ denotes the norm induced by $g$.

\begin{proposition}\label{Principal_symbol}
For any $h\in\Omega^0_+(\operatorname{Sym}^2(\Lambda^+))$, any point $x\in M$ and any non-zero cotangent vector $\xi\in T_x^*M$, the principal symbol of $D_h\Phi$ is the linear map
$$
\sigma_\xi(D_h\Phi)\colon\operatorname{Sym}^2(\Lambda_x^+)\rightarrow\operatorname{Sym}^2(\Lambda_x^+)
$$
given by
$$
\sigma_\xi(D_h\Phi)(\eta) =(\xi\wedge\operatorname{ad}_{h_x}^{-1}(\xi\wedge B_\eta))^+h_x^{-1}
$$
for every self-adjoint endomorphism $\eta\in\operatorname{Sym}^2(\Lambda_x^+)$ of the fibre $\Lambda_x^+$.
The symbol $\sigma_\xi(D_h\Phi)$ is an automorphism.
Moreover, at the identity endomorphism field $h=\operatorname{Id}$, it holds that
$$
\sigma_\xi(D_{\operatorname{Id}}\Phi)(\eta) =\frac{|\xi|^2(\eta_1-2\eta_0)}{2\sqrt{2}}.
$$
\end{proposition}

In the formula for $\sigma_\xi(D_h\Phi)(\eta)$ in Proposition~\ref{Principal_symbol}, exterior multiplication by $\xi$ and projection to the self-dual part act on the ordinary differential-form factor. The self-dual tensor $(\xi\wedge\operatorname{ad}_{h_x}^{-1}(\xi\wedge B_\eta))^+$ is regarded as an endomorphism of $\Lambda_x^+$ before composition with $h_x^{-1}$.

\begin{proof}[Proof of Proposition~\ref{Principal_symbol}]
We use the symbol convention in which each coordinate derivative $\partial_j$ is replaced by $\sqrt{-1}\xi_j$, where $\xi_j$ is the corresponding coordinate component of $\xi$. The terms containing two derivatives of $\eta$ in equation~(\ref{Linearisation_Phi}) give the asserted formula for $\sigma_\xi(D_h\Phi)(\eta)$.

Fix $x\in M$ and a non-zero $\xi\in T_x^*M$. For the remainder of this pointwise calculation, write $h$ for the positive self-adjoint endomorphism $h_x$ of $\Lambda_x^+$. Suppose that $\eta\in\operatorname{Sym}^2(\Lambda_x^+)$ satisfies $\sigma_\xi(D_h\Phi)(\eta)=0$, and put
$$
a:=\operatorname{ad}_h^{-1}(\xi\wedge B_\eta)\in T_x^*M\otimes\Lambda_x^+.
$$
Since $h^{-1}$ is invertible, we have $(\xi\wedge a)^+=0$. Choose an orthonormal basis $(e_1,e_2,e_3)$ of the coefficient space $\Lambda_x^+$, and write
$$
a=\sum_{i=1}^3a_i\otimes e_i,
$$
where each $a_i\in T_x^*M$ is an ordinary covector. Each coefficient $2$-form $\xi\wedge a_i$ is anti-self-dual and decomposable. A decomposable $2$-form has square zero, whereas an anti-self-dual $2$-form $\alpha\in\Lambda_x^-$ satisfies
$$
\alpha\wedge\alpha=-|\alpha|^2(d\mu_g)_x.
$$
It follows that $\xi\wedge a_i=0$ for every $i\in\{1,2,3\}$, and hence $\xi\wedge a=0$. Since $\xi\neq0$, every $a_i$ is a scalar multiple of $\xi$. Thus, $a=\xi\otimes f$ for a vector $f\in\Lambda_x^+$. Applying $\operatorname{ad}_h$ to the equality $a=\xi\otimes f$ gives
$$
\xi\wedge(B_\eta-[f,B_h])=0.
$$
The tensor $B_\eta-[f,B_h]$ belongs to $\Lambda_x^+\otimes\Lambda_x^+$. Exterior multiplication by a non-zero covector $\xi$ is injective on the ordinary self-dual factor $\Lambda_x^+$. Indeed, if $\alpha\in\Lambda_x^+$ satisfies $\xi\wedge\alpha=0$, then $\alpha$ is divisible by $\xi$ and has square zero, while self-duality gives $\alpha\wedge\alpha=|\alpha|^2(d\mu_g)_x$. Applying this injectivity to each coefficient of $B_\eta-[f,B_h]$, we obtain $B_\eta=[f,B_h]$.

Let $S_f\colon\Lambda_x^+\rightarrow\Lambda_x^+$ be the skew-adjoint endomorphism defined by $S_f(v)=[f,v]$ for $v\in\Lambda_x^+$. Under the correspondence in equation~(\ref{form_from_endomorphism}), the tensor $[f,B_h]$ represents the endomorphism $hS_f^*=-hS_f$, where $S_f^*$ is the adjoint of $S_f$ for the fibre metric on $\Lambda_x^+$ induced by $g$. Thus, $\eta=-hS_f$. Since $\eta$ is self-adjoint, taking adjoints gives $\eta=S_fh$, and hence
$$
hS_f+S_fh=0.
$$
Choose an orthonormal eigenbasis $(e_1,e_2,e_3)$ of $\Lambda_x^+$ for $h$, and let $\lambda_i>0$ be the eigenvalue defined by $h(e_i)=\lambda_i e_i$. Denote by $(S_f)_{ij}:=\langle S_f(e_j),e_i\rangle$ the matrix entries of $S_f$ in the eigenbasis $(e_1,e_2,e_3)$ of $\Lambda_x^+$. The identity $hS_f+S_fh=0$ gives
$$
(\lambda_i+\lambda_j)(S_f)_{ij}=0
$$
for every $i,j\in\{1,2,3\}$. Since $\lambda_i+\lambda_j>0$, all matrix entries of $S_f$ vanish. The cross-product identification of $\Lambda_x^+$ with $\mathfrak{so}(\Lambda_x^+)$ is injective, so $S_f=0$ implies $f=0$, and therefore $\eta=0$. Thus, $\sigma_\xi(D_h\Phi)$ is injective. The domain and target of $\sigma_\xi(D_h\Phi)$ are both the six-dimensional vector space $\operatorname{Sym}^2(\Lambda_x^+)$, so $\sigma_\xi(D_h\Phi)$ is an automorphism.

It remains to compute $\sigma_\xi(D_{\operatorname{Id}}\Phi)$. This principal symbol is homogeneous of degree two in $\xi$, so we first assume that $|\xi|=1$. Choose a positively oriented orthonormal basis $(\theta_0,\theta_1,\theta_2,\theta_3)$ of $T_x^*M$ such that $\theta_0=\xi$. Let $(\Sigma_1,\Sigma_2,\Sigma_3)$ be the corresponding oriented orthonormal basis of $\Lambda_x^+$ defined by equation~(\ref{self_dual_basis}). Write $\eta_{ij}:=\langle\eta(\Sigma_j),\Sigma_i\rangle$, and expand
$$
a:=\operatorname{ad}_{\operatorname{Id}}^{-1}(\xi\wedge B_\eta) =\sum_{i=1}^3a_i\otimes\Sigma_i,
$$
where $a_i\in T_x^*M$ are the ordinary coefficient $1$-forms of $a$. The equation $[a,B_{\operatorname{Id}}]=\xi\wedge B_\eta$ has the coefficient solution
$$
a_i=\sum_{j=1}^3\left(\eta_{ji}-\frac{\operatorname{tr}(\eta)}{2}\delta_{ji}\right)\theta_j
$$
for $i\in\{1,2,3\}$, where $\delta_{ji}$ is the Kronecker symbol. Indeed, substituting these coefficient $1$-forms into $[a,B_{\operatorname{Id}}]=\xi\wedge B_\eta$ verifies the equality, and Lemma~\ref{Adjoint_action} gives uniqueness. The basis convention in equation~(\ref{self_dual_basis}) gives $(\theta_0\wedge\theta_j)^+=-\Sigma_j/\sqrt{2}$. Taking the self-dual part of $\xi\wedge a$ and then regarding $(\xi\wedge a)^+$ as an endomorphism of $\Lambda_x^+$ by equation~(\ref{curvature_homomorphism}) therefore gives
$$
\sigma_\xi(D_{\operatorname{Id}}\Phi)(\eta) =-\frac{1}{\sqrt{2}}\left(\eta-\frac{\operatorname{tr}(\eta)}{2}\operatorname{Id}\right).
$$
Substituting $\eta=\eta_0+\eta_1$ and $\operatorname{tr}(\eta)\operatorname{Id}=3\eta_1$ gives the asserted formula when $|\xi|=1$. Homogeneity of degree two supplies the factor $|\xi|^2$ for an arbitrary non-zero cotangent vector $\xi$.
\end{proof}

\begin{theorem}\label{Elliptic_index}
The second order nonlinear differential operator
$$
\Phi\colon W^{s,p}_+(\operatorname{Sym}^2(\Lambda^+))\rightarrow W^{s-2,p}(\operatorname{Sym}^2(\Lambda^+))
$$
is elliptic. Moreover, for every $h\in\Omega^0_+(\operatorname{Sym}^2(\Lambda^+))$, the linearized operator
$$
D_h\Phi\colon W^{s,p}(\operatorname{Sym}^2(\Lambda^+)) \rightarrow W^{s-2,p}(\operatorname{Sym}^2(\Lambda^+))
$$
is Fredholm of index zero.
\end{theorem}
\begin{proof}[Proof of Theorem~\ref{Elliptic_index}]
Ellipticity of $\Phi$ means that the principal symbol of every linearization $D_h\Phi$ is invertible at every non-zero cotangent vector, as proved in Proposition~\ref{Principal_symbol}. Since $M$ is closed, elliptic theory implies that $D_h\Phi$ has closed image and finite-dimensional kernel and cokernel. Thus, $D_h\Phi$ is Fredholm. Its index is the dimension of its kernel minus the dimension of its cokernel.

For a fixed $h\in\Omega^0_+(\operatorname{Sym}^2(\Lambda^+))$, consider the path
$$
h(t):=(1-t)\operatorname{Id}+th
$$
for $0\leq t\leq1$. The positive cone $\Omega^0_+(\operatorname{Sym}^2(\Lambda^+))$ is convex, so $h(t)$ is a positive self-adjoint endomorphism field of $\Lambda^+$ for every $t\in[0,1]$. Proposition~\ref{Principal_symbol} shows that every operator $D_{h(t)}\Phi$ is elliptic. Since $D_{h(t)}\Phi$ depends continuously on $t$ in the operator norm from $W^{s,p}(\operatorname{Sym}^2(\Lambda^+))$ to $W^{s-2,p}(\operatorname{Sym}^2(\Lambda^+))$, the Fredholm index of $D_{h(t)}\Phi$ is independent of $t$.

To compute the index at $h=\operatorname{Id}$, we let $T$ be the bundle automorphism of $\operatorname{Sym}^2(\Lambda^+)$ defined via
$$
T(\eta):=-\sqrt{2}\eta_0+2\sqrt{2}\eta_1,
$$
where $\eta_0$ and $\eta_1$ here are respectively the trace-free and scalar parts of the symmetric endomorphism $\eta$ of $\Lambda^+$. Proposition~\ref{Principal_symbol} gives
$$
\sigma_\xi(TD_{\operatorname{Id}}\Phi)(\eta)=|\xi|^2\eta.
$$
Let $\nabla$ be the metric connection on $\operatorname{Sym}^2(\Lambda^+)$ induced by the Levi--Civita connection of $g$. Denote by $\nabla^*$ the formal $L^2$ adjoint of $\nabla$ for the volume form $d\mu_g$ and the induced bundle metrics. The connection Laplacian $\nabla^*\nabla$ has the same principal symbol $|\xi|^2\operatorname{Id}$ as $TD_{\operatorname{Id}}\Phi$. In the symbol $|\xi|^2\operatorname{Id}$ at $\xi\in T_x^*M$, the identity acts on the vector space $\operatorname{Sym}^2(\Lambda_x^+)$. Therefore, the difference $TD_{\operatorname{Id}}\Phi-\nabla^*\nabla$ has order at most one and is compact as an operator from $W^{s,p}(\operatorname{Sym}^2(\Lambda^+))$ to $W^{s-2,p}(\operatorname{Sym}^2(\Lambda^+))$. The operator $\nabla^*\nabla$ is formally self-adjoint and has index zero on the closed manifold $M$. Invariance of the Fredholm index under compact perturbations gives index zero for $TD_{\operatorname{Id}}\Phi$. Since $T$ is an invertible bundle map, composition with $T$ does not change the index, and hence $D_{\operatorname{Id}}\Phi$ also has index zero. Constancy along the path $h(t)$ now gives $\operatorname{ind}(D_h\Phi)=0$.
\end{proof}

\begin{remark}
Theorem~\ref{Elliptic_index} concerns the linearized equation $D_h\Phi(\eta)=0$ and does not assert existence or compactness of the solution set of $\Phi(h)=0$. A solution $h$ of $\Phi(h)=0$ is called non-degenerate if the kernel of $D_h\Phi$ on $W^{s,p}(\operatorname{Sym}^2(\Lambda^+))$ is zero. Since $D_h\Phi$ has index zero, the linearized operator at a non-degenerate solution $h$ of $\Phi(h)=0$ is an isomorphism from $W^{s,p}(\operatorname{Sym}^2(\Lambda^+))$ to $W^{s-2,p}(\operatorname{Sym}^2(\Lambda^+))$. The implicit function theorem therefore implies that such an $h$ is isolated among the solutions of $\Phi(h)=0$ in $W^{s,p}_+(\operatorname{Sym}^2(\Lambda^+))$. A global count of solutions of $\Phi(h)=0$ would additionally require properness of $\Phi$ on $W^{s,p}_+(\operatorname{Sym}^2(\Lambda^+))$, namely compactness of the inverse image of every compact subset of $W^{s-2,p}(\operatorname{Sym}^2(\Lambda^+))$.
\end{remark}

We next use the ellipticity of $\Phi$ to establish regularity for weak solutions of $\Phi(h)=0$. A weak solution of $\Phi(h)=0$ is a section $h\in W^{s,p}_+(\operatorname{Sym}^2(\Lambda^+))$ for which the equation holds in the sense of distributions on $M$.

\begin{corollary}\label{Elliptic_regularity}
Every weak solution $h\in W^{s,p}_+(\operatorname{Sym}^2(\Lambda^+))$ of $\Phi(h)=0$ is in fact a smooth section of $\operatorname{Sym}^2(\Lambda^+)$.
\end{corollary}
\begin{proof}[Proof of Corollary~\ref{Elliptic_regularity}]
Since $s\geq2$, we have that $h\in W^{2,p}(\operatorname{Sym}^2(\Lambda^+))$. Choose smooth local coordinates $(x_1,x_2,x_3,x_4)$ on an open subset of $M$, and choose a smooth orthonormal frame $(e_1,e_2,e_3)$ of the restriction of $\Lambda^+$ to this subset. In this frame, write $h$ as the symmetric matrix with entries $h_{ab}:=\langle h(e_b),e_a\rangle$, for $1\leq a,b\leq3$. In the local calculations below, $\partial_i h$ and $\partial_i\partial_j h$ denote the matrices obtained by differentiating these entries with respect to the coordinates $x_i$ and $x_j$. We use $\partial h$ to denote the tuple $(\partial_1h,\partial_2h,\partial_3h,\partial_4h)$.

In the coordinates $(x_1,x_2,x_3,x_4)$ and the local frame $(e_1,e_2,e_3)$ of $\Lambda^+$, substituting the expression defining $A(h)$ into the curvature $F_{A(h)}$ allows us to write the equation $\Phi(h)=0$ as
$$
\sum_{i,j=1}^4a^{ij}(x,h)\partial_i\partial_jh+b(x,h,\partial h)=0,
$$
where each $a^{ij}(x,h)$ is a linear map from the space $\operatorname{Sym}^2(\mathbb R^3)$ of symmetric real $3\times3$ matrices to itself, and $b(x,h,\partial h)$ takes values in the same space. The functions $a^{ij}$ and $b$ include the smooth coefficients of the chosen coordinates, frame and reference connection $\Gamma$. Each $a^{ij}$ depends smoothly on $x$ and the positive matrix $h$, and $b$ depends smoothly on $x$ and $h$ and is polynomial of degree at most two in the entries of $\partial h$. In particular, the coefficients of the second derivatives depend on $h$.

Since $p>4$, Sobolev embedding gives $h\in C^{1,\alpha}(\operatorname{Sym}^2(\Lambda^+))$, where $\alpha:=1-4/p$. Since $h$ is positive and $M$ is compact, the smallest eigenvalue of $h_x\colon\Lambda_x^+\rightarrow\Lambda_x^+$ has a positive lower bound independent of $x\in M$. Proposition~\ref{Principal_symbol} therefore gives a uniformly invertible principal symbol for the linear operator
$$
L_hv:=\sum_{i,j=1}^4a^{ij}(x,h(x))\partial_i\partial_jv
$$
on each coordinate ball whose closure is contained in the chosen coordinate neighbourhood. Here, $v$ is a function taking values in $\operatorname{Sym}^2(\mathbb R^3)$. The coefficients of $L_h$ are $C^{1,\alpha}$, so the interior $W^{2,p}$ estimate for elliptic systems applies to $L_h$, by freezing its coefficients on sufficiently small coordinate balls.

For the interior $W^{2,p}$ estimate for $L_h$, choose concentric coordinate balls $B_r\subseteq B_{2r}\subseteq B_{3r}$ whose closures are contained in the coordinate neighbourhood. Fix a coordinate direction $x_k$, with $1\leq k\leq4$. For a non-zero real number $t$ of sufficiently small absolute value, define the difference quotient
$$
\delta_{t,k}h(x):=\frac{h(x+t\hat e_k)-h(x)}{t},
$$
where $\hat e_k$ is the $k$-th standard coordinate vector of $\mathbb R^4$. We use $\delta_{t,k}$ also for difference quotients of the other coefficient functions in this coordinate neighbourhood. In particular, set
$$
\begin{cases}
a_h^{ij}(x):=a^{ij}(x,h(x)),\\
b_h(x):=b(x,h(x),\partial h(x)).
\end{cases}
$$
Taking the difference quotient of the local expression for $\Phi(h)=0$ gives
$$
L_h(\delta_{t,k}h)(x) =-\sum_{i,j=1}^4\delta_{t,k}a_h^{ij}(x)\partial_i\partial_jh(x+t\hat e_k)-\delta_{t,k}b_h(x).
$$
The functions $\delta_{t,k}a_h^{ij}$ are uniformly bounded on $B_{2r}$ because $h$ is $C^{1,\alpha}$. The second derivatives of $h$ belong to $L^p(B_{3r})$. The chain rule for $b$ shows that $\delta_{t,k}b_h$ is bounded in $L^p(B_{2r})$ independently of $t$: its terms contain bounded coefficients multiplied by $1$, $\delta_{t,k}h$ or $\delta_{t,k}\partial_i h$, and these difference quotients are bounded in $L^p(B_{2r})$ because $h\in W^{2,p}$. Consequently, the interior $W^{2,p}$ estimate for $L_h$ gives
$$
\|\delta_{t,k}h\|_{W^{2,p}(B_r)} \leq C(\|L_h(\delta_{t,k}h)\|_{L^p(B_{2r})}+\|\delta_{t,k}h\|_{L^p(B_{2r})}) \leq C',
$$
where $C$ and $C'$ are independent of $t$. All norms in this local estimate are the ordinary Sobolev norms of the matrix entries in the chosen coordinates and frame. The difference-quotient characterization of weak derivatives now gives $h\in W^{3,p}(B_r)$. Since the coordinate neighbourhood and the direction $x_k$ were arbitrary, we obtain that $h\in W^{3,p}(\operatorname{Sym}^2(\Lambda^+))$.

Repeating the interior $W^{2,p}$ estimate for $L_h$ after differentiating the local expression for $\Phi(h)=0$ gives
$$
h\in W^{m,p}(\operatorname{Sym}^2(\Lambda^+))
$$
for every integer $m\geq2$. At each step, the coefficients depend smoothly on $h$, and multiplication of $W^{1,p}$ functions preserves $W^{1,p}$ because $p>4$. Sobolev embedding therefore proves that $h$ is smooth.
\end{proof}

We now return to the functional $E$ and compute its second variation at a solution of $\Phi(h)=0$. For later use, we introduce the following notation: For a variation $\phi\in\Omega^0(\operatorname{Sym}^2(\Lambda^+))$, the form $B_\phi$ is given, in a local oriented orthonormal frame $(\Sigma_1,\Sigma_2,\Sigma_3)$ of $\Lambda^+|_{U}$ over an open subset $U\subseteq M$, by
$$
B_\phi|_{U}=\sum_{i=1}^3\phi(\Sigma_i)\otimes\Sigma_i.
$$
Thus, the ordinary self-dual $2$-form $\phi(\Sigma_i)$ is the coefficient of $B_\phi$ along the section $\Sigma_i$ of the coefficient bundle $\Lambda^+|_{U}$.

\begin{proposition}\label{Second_variation}
Assume that $M$ is closed and let $h$ be a solution of $\Phi(h)=0$. The second variation $\delta^2E[h]$ of the functional $E$ at $h$ satisfies
$$
\delta^2E[h](\phi,\psi) =-\int_M\langle (d_{A(h)}\operatorname{ad}_h^{-1}(d_{A(h)}B_\phi))^++\phi,\psi \rangle d\mu_g,
$$
for $\phi,\psi\in\Omega^0(\operatorname{Sym}^2(\Lambda^+))$.
In particular, the bilinear form $\delta^2E[h]$ is symmetric.
\end{proposition}

In the formula for $\delta^2E[h](\phi,\psi)$ in Proposition~\ref{Second_variation}, the superscript $+$ denotes the self-dual part of the $\Lambda^+$-valued $2$-form $d_{A(h)}\operatorname{ad}_h^{-1}(d_{A(h)}B_\phi)$. This self-dual part is regarded as an endomorphism of $\Lambda^+$ by the coefficient contraction in equation~(\ref{curvature_homomorphism}). The inner product in the formula for $\delta^2E[h](\phi,\psi)$ is the pointwise inner product of endomorphisms of $\Lambda^+$ induced by $g$.

\begin{proof}[Proof of Proposition~\ref{Second_variation}]
The first variation computed above in Proposition~\ref{E-L} is
$$
\delta E[h](\psi)=\int_M\langle F_{A(h)}^+-h,\psi\rangle d\mu_g.
$$
Differentiate the identity for $\delta E[h](\psi)$ with respect to $h$ in the direction $\phi$. For real $t$ of sufficiently small absolute value, the section $h+t\phi$ is positive; define $A(t):=A(h+t\phi)$ and $A:=A(0)$. Each $A(t)$ is a metric connection on the fixed bundle $\Lambda^+$. Lemma~\ref{Linearisation_formula} gives
$$
\left.\frac{d}{dt}\right|_{t=0}F_{A(t)}^+ =-(d_A\operatorname{ad}_h^{-1}(d_AB_\phi))^+.
$$
Substitution of this curvature derivative into the derivative of $\delta E[h](\psi)$ gives the stated expression for $\delta^2E[h](\phi,\psi)$. The equality of the mixed second derivatives of the smooth functional $E$ gives $\delta^2E[h](\phi,\psi)=\delta^2E[h](\psi,\phi)$.
\end{proof}

\begin{lemma}\label{Constant_scaling}
For every constant $\lambda>0$ and every $h\in\Omega^0_+(\operatorname{Sym}^2(\Lambda^+))$, it holds that
$$
\begin{cases}
A(\lambda h)=A(h),\\
\Phi(\lambda h)=\lambda^{-1}\Phi(h)+(\lambda^{-1}-1)\operatorname{Id}.
\end{cases}
$$
Consequently, for every fixed $h\in\Omega^0_+(\operatorname{Sym}^2(\Lambda^+))$, the ray $\{\lambda h:\lambda>0\}$ contains at most one solution of the equation $\Phi(h)=0$.
\end{lemma}
\begin{proof}[Proof of Lemma~\ref{Constant_scaling}]
Since $\lambda$ is constant, we have $B_{\lambda h}=\lambda B_h$ and $d_\Gamma B_{\lambda h}=\lambda d_\Gamma B_h$ by the definition of $B_h$. The definition of $\operatorname{ad}_h$ also gives $\operatorname{ad}_{\lambda h}=\lambda\operatorname{ad}_h$. Substitution into the definition of $A(h)$ yields
$$
A(\lambda h)=\Gamma-\lambda^{-1}\operatorname{ad}_h^{-1}(\lambda d_\Gamma B_h)=A(h).
$$
Consequently, we obtain that 
$$
\Phi(\lambda h)=F_{A(h)}^+(\lambda h)^{-1}-\operatorname{Id}=\lambda^{-1}\Phi(h)+(\lambda^{-1}-1)\operatorname{Id}.
$$
If $\Phi(h)=\Phi(\lambda h)=0$, then the displayed expression for $\Phi(\lambda h)$ implies $(\lambda^{-1}-1)\operatorname{Id}=0$. Thus, we conclude that $\lambda=1$.
\end{proof}

\subsection{The conformal subcone reduction}

We finish the section by considering the equation $\Phi_g(h)=0$ for endomorphisms of $\Lambda_g^+$ of the form $h=e^{2\omega}\operatorname{Id}$, where $\omega$ is a smooth real-valued function on $M$ and $\operatorname{Id}$ denotes the identity of $\Lambda_g^+$. Such an $h$ is an endomorphism of $\Lambda_g^+$, whereas $\hat g:=e^{2\omega}g$ is a Riemannian metric on $TM$. We use the subscript $g$ in $\Phi_g$ to emphasize the fixed background metric in the definition of $\Phi$. Proposition~\ref{Scalar_solution} identifies the connection $A(h)$ on $\Lambda_g^+$ with the connection on $\Lambda_{\hat g}^+$ induced by the Levi--Civita connection of $\hat g$, transported to $\Lambda_g^+$. The proposition also expresses $\Phi_g(h)$ in terms of the self-dual Weyl curvature and scalar curvature of $\hat g$. Recall that the metric $\hat g$ is anti-self-dual precisely when its self-dual Weyl curvature $W_{\hat g}^+$ vanishes.

\begin{proposition}\label{Scalar_solution}
For a smooth real-valued function $\omega$ on $M$, let $h:=e^{2\omega}\operatorname{Id}$ and $\hat g:=e^{2\omega}g$.
Then the positive endomorphism $h$ of $\Lambda_g^+$ solves the equation $\Phi_g(h)=0$ if and only if the metric $\hat g$ is anti-self-dual and has constant scalar curvature $6\sqrt{2}$.
\end{proposition}
\begin{proof}[Proof of Proposition~\ref{Scalar_solution}]
The Hodge star on ordinary $2$-forms is conformally invariant in dimension four. Consequently, $\Lambda_g^+$ and $\Lambda_{\hat g}^+$ are the same subbundle of $\Lambda^2$, but the fibre metrics induced by $g$ and $\hat g$ satisfy
$$
\langle\alpha,\gamma\rangle_{\hat g}=e^{-4\omega}\langle\alpha,\gamma\rangle_g
$$
for every point $x\in M$ and every pair $\alpha,\gamma\in\Lambda_{g,x}^+$. Define a bundle map $R\colon\Lambda_g^+\rightarrow\Lambda_{\hat g}^+$ by $R(\alpha):=e^{2\omega}\alpha$. The relation between the fibre metrics induced by $g$ and $\hat g$ shows that $R$ is an orientation-preserving isometry from $\Lambda_g^+$ to $\Lambda_{\hat g}^+$.

Let $\hat\nabla$ denote the metric connection on $\Lambda_{\hat g}^+$ induced by the Levi--Civita connection of $\hat g$. Define a metric connection $A$ on $\Lambda_g^+$ by
$$
\nabla^A_Xv:=R^{-1}(\hat\nabla_X(Rv))
$$
for every smooth tangent vector field $X$ on $M$ and every smooth section $v$ of $\Lambda_g^+$. Choose an open subset $U\subseteq M$ and a local oriented orthonormal frame $(\Sigma_1,\Sigma_2,\Sigma_3)$ of $\Lambda_g^+|_{U}$, and set $\hat\Sigma_i:=e^{2\omega}\Sigma_i$ for $1\leq i\leq3$. Thus, $(\hat\Sigma_1,\hat\Sigma_2,\hat\Sigma_3)$ is a local orthonormal frame of $\Lambda_{\hat g}^+|_{U}$.

Define the $\Lambda_{\hat g}^+$-valued self-dual $2$-form $\hat B$ by
$$
\hat B|_{U}:=\sum_{i=1}^3\hat\Sigma_i\otimes\hat\Sigma_i.
$$
This definition is equation~(\ref{form_from_endomorphism}) with background metric $\hat g$ and endomorphism $\operatorname{Id}_{\Lambda_{\hat g}^+}$.
In each tensor product $\hat\Sigma_i\otimes\hat\Sigma_i$ in $\hat B|_{U}$, the first factor $\hat\Sigma_i$ is regarded as an ordinary $2$-form on $U$, and the second factor $\hat\Sigma_i$ is regarded as a section of the coefficient bundle $\Lambda_{\hat g}^+|_{U}$. The $\Lambda_{\hat g}^+$-valued $2$-form $\hat B$ is parallel for the tensor product connection induced by the Levi--Civita connection of $\hat g$ on $\Lambda^2$ and by $\hat\nabla$ on $\Lambda_{\hat g}^+$. Since the Levi--Civita connection of $\hat g$ is torsion-free, alternation of the tensor product covariant derivative of $\hat B$ gives
$$
d_{\hat\nabla}\sum_{i=1}^3\hat\Sigma_i\otimes\hat\Sigma_i=0.
$$
Applying $R^{-1}$ only to the coefficient factor of the $\Lambda_{\hat g}^+$-valued $2$-form $\hat B$ gives
$$
(\operatorname{Id}_{\Lambda^2}\otimes R^{-1}) \sum_{i=1}^3\hat\Sigma_i\otimes\hat\Sigma_i =\sum_{i=1}^3\hat\Sigma_i\otimes\Sigma_i =B_h.
$$
In the displayed expression for $B_h$, the map $\operatorname{Id}_{\Lambda^2}$ fixes the ordinary $2$-form in each tensor product, while $R^{-1}$ changes the coefficient factor from a section of $\Lambda_{\hat g}^+$ to a section of $\Lambda_g^+$. For each integer $k\geq0$, extend $R^{-1}$ to a map from $\Lambda_{\hat g}^+$-valued $k$-forms to $\Lambda_g^+$-valued $k$-forms by applying $R^{-1}$ to the coefficient factor and the identity to the ordinary $k$-form factor. The definition of $\nabla^A$ implies that the maps induced by $R^{-1}$ on bundle-valued forms commute with the exterior covariant derivatives $d_{\hat\nabla}$ and $d_A$. Consequently, we have that $d_AB_h=0$, and uniqueness in Lemma~\ref{Compatible_connection} gives $A=A(h)$.

Write $W_{\hat g}^+$ for the self-dual Weyl curvature, regarded as a trace-free self-adjoint endomorphism of $\Lambda_{\hat g}^+$, and write $R_{\hat g}$ for the scalar curvature of $\hat g$. Define $W^+:=R^{-1}W_{\hat g}^+R$ as an endomorphism of $\Lambda_g^+$.
To compute the curvature $F_{\hat\nabla}$ of the connection $\hat\nabla$ on $\Lambda_{\hat g}^+$ with the normalization in equation~(\ref{self_dual_basis}), choose a local oriented $\hat g$-orthonormal co-frame $(\hat\theta_0,\hat\theta_1,\hat\theta_2,\hat\theta_3)$ of $T^*M|_{U}$ inducing $(\hat\Sigma_1,\hat\Sigma_2,\hat\Sigma_3)$, after shrinking $U$ if necessary, with dual frame $(\hat E_0,\hat E_1,\hat E_2,\hat E_3)$ of $TM|_{U}$. Let $\hat R$ denote the Riemannian curvature tensor of $\hat g$, with the curvature convention used throughout this paper; the tensor $\hat R$ is distinct from the scalar function $R_{\hat g}$. For $0\leq a,b\leq3$, define an ordinary $2$-form $\Omega_{ab}$ by
$$
\Omega_{ab}(X,Y):=\langle\hat R(X,Y)\hat E_b,\hat E_a\rangle_{\hat g}
$$
for tangent vector fields $X$ and $Y$ on $U$.

Use the cross-product identification of $\mathfrak{so}(\Lambda_{\hat g}^+)$ with $\Lambda_{\hat g}^+$ to write the curvature of $\hat\nabla$ as
$$
F_{\hat\nabla}=\sum_{i=1}^3F^i\otimes\hat\Sigma_i,
$$
where each $F^i$ is an ordinary $2$-form on $U$ and $\hat\Sigma_i$ is its coefficient section of $\Lambda_{\hat g}^+|_{U}$. The connection $\hat\nabla$ induced on $\Lambda_{\hat g}^+$ gives
$$
F^i=-(\Omega_{0i}+\Omega_{jk}),
$$
where $(i,j,k)$ is a cyclic permutation of $(1,2,3)$. The normalization of the local orthonormal frame $(\hat\Sigma_1,\hat\Sigma_2,\hat\Sigma_3)$ of $\Lambda_{\hat g}^+|_{U}$ is
$$
\hat\Sigma_i=-\frac{\hat\theta_0\wedge\hat\theta_i+\hat\theta_j\wedge\hat\theta_k}{\sqrt{2}}.
$$
Define $W^+_{ij}:=\langle W_{\hat g}^+(\hat\Sigma_j),\hat\Sigma_i\rangle_{\hat g}$ for $1\leq i,j\leq3$, so that $(W^+_{ij})$ is the matrix of $W_{\hat g}^+$ in the frame $(\hat\Sigma_1,\hat\Sigma_2,\hat\Sigma_3)$. The self-dual curvature decomposition \cite{ahs,besse} then gives
$$
(F^i)^+=\sqrt{2}\sum_{j=1}^3\left(W^+_{ij}+\frac{R_{\hat g}}{12}\delta_{ij}\right)\hat\Sigma_j,
$$
where $(F^i)^+$ denotes the self-dual part of the ordinary $2$-form $F^i$, and $\delta_{ij}$ is the Kronecker symbol. The projection to self-dual $2$-forms is the same for $g$ and $\hat g$. Applying $R^{-1}$ to the coefficient factor of $F_{\hat\nabla}$ gives the curvature of $A(h)$ on $\Lambda_g^+$ as
$$
F_{A(h)}=\sum_{i=1}^3F^i\otimes\Sigma_i.
$$
Since $\hat\Sigma_j=e^{2\omega}\Sigma_j$, the identification in equation~(\ref{curvature_homomorphism}) using the fixed metric $g$ therefore gives
$$
F_{A(h)}^+=\sqrt{2}e^{2\omega}\left(W^++\frac{R_{\hat g}}{12}\operatorname{Id}\right).
$$
In the displayed formula for $F_{A(h)}^+$, both $F_{A(h)}^+$ and $W^+$ are self-adjoint endomorphisms of $\Lambda_g^+$, and $\operatorname{Id}$ is the identity of $\Lambda_g^+$. Since $h^{-1}=e^{-2\omega}\operatorname{Id}$, the definition $\Phi_g(h)=F_{A(h)}^+h^{-1}-\operatorname{Id}$ yields
$$
\Phi_g(h)=\sqrt{2}W^++\left(\frac{\sqrt{2}R_{\hat g}}{12}-1\right)\operatorname{Id}.
$$
The endomorphism $W^+$ of $\Lambda_g^+$ is trace-free and $\Lambda_g^+$ has rank three. Taking the trace of the displayed expression for $\Phi_g(h)$ shows that $\Phi_g(h)=0$ implies $R_{\hat g}=6\sqrt{2}$. Substituting this scalar curvature value into the expression for $\Phi_g(h)$ gives $W^+=0$. Since $W^+=R^{-1}W_{\hat g}^+R$, the equality $W^+=0$ is equivalent to $W_{\hat g}^+=0$. Conversely, substitution of $R_{\hat g}=6\sqrt{2}$ and $W_{\hat g}^+=0$ into the expression for $\Phi_g(h)$ gives $\Phi_g(h)=0$.
\end{proof}

\begin{remark}
The normalization $F_{A(h)}^+=h$ is stronger than the Yang--Mills equation alone. For the round metric $g:=g(r)$ of radius $r$ on the oriented sphere $\mathbb{S}^4$, its induced Levi--Civita connection $\Gamma:=\Gamma(r)$ on $\Lambda_{g}^+\rightarrow \mathbb{S}^4$ is Yang--Mills for every $r>0$, and its self-dual curvature satisfies
$$
F_{\Gamma}^+=\frac{\sqrt{2}}{r^2}\operatorname{Id}.
$$
Here, the tensor $F_{\Gamma}^+$ is regarded as a positive self-adjoint endomorphism of the Euclidean bundle $\Lambda_g^+\rightarrow\mathbb{S}^4$. This identification contracts the coefficient factor of the $\Lambda_g^+$-valued self-dual $2$-form $F_{\Gamma}^+$ using the fibre metric induced by $g$, as in equation~(\ref{curvature_homomorphism}). The symbol $\operatorname{Id}$ denotes the identity endomorphism of $\Lambda_g^+$. Thus, the positive self-adjoint endomorphism $h:=(\sqrt{2}/r^2)\operatorname{Id}$ of $\Lambda_g^+$ is the solution of $\Phi_g(h)=0$ corresponding to $\Gamma$, and it satisfies $A(h)=\Gamma$.
The conformal metric on $\mathbb{S}^4$ associated with the scalar endomorphism $h$ of $\Lambda_g^+$ is $\hat g=(\sqrt{2}/r^2)g$, whose scalar curvature is $6\sqrt{2}$.
The particular endomorphism $\operatorname{Id}$ of $\Lambda_g^+$ satisfies $\Phi_g(\operatorname{Id})=0$ only when $r^2=\sqrt{2}$. As in Lemma~\ref{Constant_scaling}, for each fixed background metric $g$, constant rescaling of a positive endomorphism of $\Lambda_g^+$ preserves its compatible connection. For the solution $h=(\sqrt{2}/r^2)\operatorname{Id}$ of $\Phi_g(h)=0$, the equation $\Phi_g(\lambda h)=0$ with $\lambda>0$ holds only when $\lambda=1$. Thus, the ray $\{\lambda h:\lambda>0\}$ contains exactly one solution of the frame equation $\Phi_g(h)=0$, namely $h$.
\end{remark}

\begin{remark}\label{Conformal_comparison}
The use of a conformally rescaled Levi--Civita connection is classical. Etesi and Hausel \cite{etesihausel} construct instantons on Taub--NUT space by conformally rescaling its metric using positive harmonic functions and projecting the spin connection to one of the two $\mathfrak{su}(2)$ factors of $\mathfrak{spin}(4)$. Their rescaled metric is half-conformally flat, meaning that one of its two Weyl curvature components vanishes, and scalar-flat. For a smooth real-valued function $\omega$ on $M$, suppose that $h:=e^{2\omega}\operatorname{Id}$ satisfies $\Phi_g(h)=0$. By Proposition~\ref{Scalar_solution}, the rescaled metric $\hat g=e^{2\omega}g$ satisfies $W_{\hat g}^+=0$ and $R_{\hat g}=6\sqrt{2}$. The connection $A(h)$ on $\Lambda_g^+$, obtained by transporting the connection induced by $\hat g$ on $\Lambda_{\hat g}^+$, consequently satisfies $F_{A(h)}^+=h$ as an equality of endomorphisms of $\Lambda_g^+$. The anti-self-dual curvature $F_{A(h)}^-\in\Omega^0(\Lambda_g^-\otimes\Lambda_g^+)$ is determined by the trace-free Ricci tensor of $\hat g$ and need not vanish. Here, the trace-free Ricci tensor is $\operatorname{Ric}_{\hat g}-(R_{\hat g}/4)\hat g$, where $\operatorname{Ric}_{\hat g}$ is the Ricci tensor of $\hat g$. Thus, the Yang--Mills connections $A(h)$ on $\Lambda_g^+$ obtained in Proposition~\ref{Scalar_solution} are self-dual instantons, meaning $F_{A(h)}^-=0$, precisely when $\hat g$ is Einstein.
\end{remark}

Combining Proposition~\ref{Scalar_solution} with the solution of the Yamabe problem gives the following existence theorem for $\Phi_g(h)=0$.

\begin{theorem}\label{Yamabe_problem}
Let $(M,g)$ be a closed oriented anti-self-dual Riemannian 4-manifold whose conformal class $[g]$ has positive Yamabe constant.
Then the equation $\Phi_g(h)=0$ admits a solution. More precisely, if $\hat g=e^{2\omega}g$ is a Yamabe metric of constant scalar curvature $6\sqrt{2}$, then $h=e^{2\omega}\operatorname{Id}$ is a solution to $\Phi_g(h)=0$.
\end{theorem}
\begin{proof}[Proof of Theorem~\ref{Yamabe_problem}]
By the solution of the Yamabe problem, the conformal class $[g]$ contains a smooth metric of positive constant scalar curvature. A further positive constant rescaling gives a Yamabe metric $\hat g=e^{2\omega}g$ of scalar curvature $6\sqrt{2}$, for a smooth real-valued function $\omega$ on $M$.
Since anti-self-duality is conformally invariant in dimension four, the metric $\hat g$ is anti-self-dual.
Proposition~\ref{Scalar_solution} therefore gives $\Phi_g(e^{2\omega}\operatorname{Id})=0$.
\end{proof}
\section{Concluding Remarks}

The formalism developed in this article gives a global description of the Yang--Mills equation on the open subset of the moduli space of Yang--Mills connections on $\Lambda^+$ for which the self-dual curvature homomorphism $F_A^+\colon\Lambda^+\rightarrow\Lambda^+$ is an orientation-preserving isomorphism. Theorem~\ref{Yang_Mills_equivalence} identifies this subset with the solution set of $\Phi_g(h)=0$. The unknown $h$ is a positive self-adjoint endomorphism field of $\Lambda^+$, and the value $\Phi_g(h)$ is a section of the same bundle $\operatorname{Sym}^2(\Lambda^+)$ of rank six. The polar decomposition of $F_A^+$ gives a unique representative of each gauge class whose self-dual curvature is a positive self-adjoint endomorphism of $\Lambda^+$, before the equation $\Phi_g(h)=0$ is studied. Theorem~\ref{Elliptic_index} then shows that every linearization $D_h\Phi_g$ is elliptic and has index zero. Thus, the equation $\Phi_g(h)=0$ provides a setting for local moduli theory and variational arguments directly in terms of the endomorphism field $h$.

For an anti-self-dual conformal class of positive Yamabe constant, Theorem~\ref{Yamabe_problem} gives a global solution of $\Phi_g(h)=0$ of the form $h=e^{2\omega}\operatorname{Id}$. The positive endomorphisms of $\Lambda^+$ of the form $e^{2\omega}\operatorname{Id}$ constitute the scalar part of $\Omega^0_+(\operatorname{Sym}^2(\Lambda^+))$. A basic problem is to determine whether solutions of $\Phi_g(h)=0$ which are not scalar multiples of $\operatorname{Id}$ exist and, if so, whether families of such solutions can meet the scalar family. To study this question, one should compute $D_h\Phi_g$ explicitly at a solution $h=e^{2\omega}\operatorname{Id}$ of $\Phi_g(h)=0$ for which $e^{2\omega}g$ is an anti-self-dual Yamabe metric. One should then compare the linearization $D_h\Phi_g$ at that scalar solution of $\Phi_g(h)=0$ with the deformation complex for anti-self-dual conformal structures. If the kernel of $D_h\Phi_g$ is zero, the index-zero theorem makes $D_h\Phi_g$ invertible, and the implicit function theorem gives persistence of solutions of $\Phi_g(h)=0$ under sufficiently small smooth changes of the background metric $g$, after identifying the corresponding bundles of self-dual forms.

The main global difficulty is compactness of the solution set of $\Phi_g(h)=0$. For a sequence of solutions $h_n$ of $\Phi_g(h_n)=0$, the associated Yang--Mills connections $A(h_n)$ may develop bubbles through concentration of curvature. The smallest eigenvalue of $h_n$ may also tend to zero, so that the sequence approaches the boundary of the positive cone $\Omega^0_+(\operatorname{Sym}^2(\Lambda^+))$. Since $F_{A(h_n)}^+=h_n$ as endomorphisms of $\Lambda^+$, this second possibility is the degeneration of the curvature homomorphism on which the present description depends. A degree theory for $\Phi_g$ would require a priori estimates controlling both curvature concentration and loss of positivity. Corollary~\ref{Energy_identity} expresses the reduced energy $E[h]$ of a solution $h$ of $\Phi_g(h)=0$ in terms of $|h|^2$. For any positive endomorphism field $h$ of $\Lambda^+$, one may also write
$$
h=(\det h)^{1/3}((\det h)^{-1/3}h),
$$
where $\det h$ is the positive function obtained by taking the determinant of $h_x$ at each $x\in M$. The positive function $(\det h)^{1/3}$ measures the scalar size of $h$, and the positive endomorphism $(\det h)^{-1/3}h$ of $\Lambda^+$ has determinant one. This decomposition of $h$ may help distinguish control of the scalar size from control of the eigenvalue ratios.

One should also compare the linearization $D_h\Phi_g$ with the Yang--Mills Jacobi operator at $A(h)$, namely the linearization of the Yang--Mills equation with respect to the connection, and with the Yang--Mills detour complex \cite{gover}. Such a comparison should relate the infinitesimal solutions of $D_h\Phi_g(\eta)=0$ to infinitesimal Yang--Mills deformations modulo infinitesimal gauge transformations. The comparison should also relate the determinant lines associated with $D_h\Phi_g$ and with the Yang--Mills deformation complex. For a Fredholm operator, its determinant line is the tensor product of the top exterior power of its kernel and the dual of the top exterior power of its cokernel. This would place the index calculation of Theorem~\ref{Elliptic_index} within the usual deformation theory of Yang--Mills connections.

A further question concerns the Yang--Mills stress-energy tensor, obtained by varying the Yang--Mills energy with respect to the background metric. In dimension four, the Yang--Mills stress-energy tensor is a trace-free symmetric covariant $2$-tensor on $M$, and it is invariant under gauge transformations of the connection. Thus, assigning to a gauge class represented by $A(h)$ its stress-energy tensor defines a map from the moduli space described in Theorem~\ref{Yang_Mills_equivalence} to the space of trace-free symmetric covariant $2$-tensors on $M$. Asking which prescribed tensors arise in this way gives a Rainich-type problem. The equation $\Phi_g(h)=0$ describes the Yang--Mills connection before any stress-energy tensor is prescribed.

The hypothesis that $F_A^+\colon\Lambda^+\rightarrow\Lambda^+$ is an orientation-preserving isomorphism is essential to the construction of the positive endomorphism $h$ of $\Lambda^+$ and to the reconstruction of the connection $A(h)$ from the $\Lambda^+$-valued $2$-form $B_h$. Under this hypothesis, the Yang--Mills equation becomes the determined elliptic system $$\Phi_g(h)\equiv F_{A(h)}^+h^{-1}-\operatorname{Id}=0.$$ Understanding the boundary of this subset of the Yang--Mills moduli space, and whether a useful compactification can include limits for which the curvature homomorphism loses rank, remains a further problem.

\section*{Acknowledgement}

The author is deeply grateful to Alexander Veselov and Evgeny Ferapontov for very useful and inspiring discussions.
Special thanks should also go to Cheng He and Zongjian Han for careful proofreading. The author expresses gratitude to the reviewers for various suggestions as well.

The author is indebted to the Russian mathematical society, especially to all the professors teaching in the ``Math in Moscow'' program, and most importantly to the author's supervisor Alexander Petrovich Veselov at Loughborough University, as they cultivated the author's mathematical literacy and maturity.

This research was completed while the author was studying at the Mathematics Institute of the University of Warwick.
The author therefore would like to thank the University of Warwick for its hospitality.

\section*{Statements and Declarations}

\noindent\textbf{Funding.} No funding was received to assist with the preparation of this manuscript.

\noindent\textbf{Competing interests.} The author declares no relevant financial or non-financial interests.

\noindent\textbf{Data availability.} Data sharing is not applicable to this article.

\noindent\textbf{Use of generative AI.} During the preparation of this manuscript, the author used OpenAI’s ChatGPT to help handle minor details and grammatical issues. All suggested changes were reviewed by the author, who takes full responsibility for the mathematical accuracy and the final content of this article.


\begin{thebibliography}{99}

\bibitem{ahs}
M.~F. Atiyah, N.~J. Hitchin and I.~M. Singer,
\emph{Self-duality in four-dimensional Riemannian geometry},
Proc. Roy. Soc. London Ser. A \textbf{362} (1978), 425--461.

\bibitem{aubin}
T. Aubin,
\emph{\'{E}quations diff\'{e}rentielles non lin\'{e}aires et probl\`eme de Yamabe concernant la courbure scalaire},
J. Math. Pures Appl. \textbf{55} (1976), 269--296.

\bibitem{bpst}
A.~A. Belavin, A.~M. Polyakov, A.~S. Schwartz and Yu.~S. Tyupkin,
\emph{Pseudoparticle solutions of the Yang--Mills equations},
Phys. Lett. B \textbf{59} (1975), 85--87.

\bibitem{besse}
A.~L. Besse,
\emph{Einstein Manifolds},
Ergebnisse der Mathematik und ihrer Grenzgebiete, vol.~10, Springer, 1987.

\bibitem{bourguignonlawson}
J.-P. Bourguignon and H.~B. Lawson, Jr.,
\emph{Stability and isolation phenomena for Yang--Mills fields},
Comm. Math. Phys. \textbf{79} (1981), 189--230.

\bibitem{capovilla}
R. Capovilla, T. Jacobson and J. Dell,
\emph{General relativity without the metric},
Phys. Rev. Lett. \textbf{63} (1989), 2325--2328.

\bibitem{chalmers}
G. Chalmers and W. Siegel,
\emph{The self-dual sector of QCD amplitudes},
Phys. Rev. D \textbf{54} (1996), 7628--7633.

\bibitem{deserteitelboim}
S. Deser and C. Teitelboim,
\emph{Duality transformations of Abelian and non-Abelian gauge fields},
Phys. Rev. D \textbf{13} (1976), 1592--1597.

\bibitem{donaldson}
S.~K. Donaldson,
\emph{An application of gauge theory to four-dimensional topology},
J. Differential Geom. \textbf{18} (1983), 279--315.

\bibitem{donaldsonforms}
S.~K. Donaldson,
\emph{Two-forms on four-manifolds and elliptic equations},
in \emph{Inspired by S.~S. Chern}, Nankai Tracts Math., vol.~11, World Scientific, 2006, 153--172.

\bibitem{donaldsonkronheimer}
S.~K. Donaldson and P.~B. Kronheimer,
\emph{The Geometry of Four-Manifolds},
Oxford Mathematical Monographs, Oxford University Press, 1990.

\bibitem{etesihausel}
G. Etesi and T. Hausel,
\emph{Geometric construction of new Yang--Mills instantons over Taub--NUT space},
Phys. Lett. B \textbf{514} (2001), 189--199.

\bibitem{fine}
J. Fine, K. Krasnov and D. Panov,
\emph{A gauge theoretic approach to Einstein $4$-manifolds},
New York J. Math. \textbf{20} (2014), 293--323.

\bibitem{freeduhlenbeck}
D.~S. Freed and K.~K. Uhlenbeck,
\emph{Instantons and Four-Manifolds},
2nd ed., Mathematical Sciences Research Institute Publications, vol.~1, Springer, 1991.

\bibitem{freidel}
L. Freidel,
\emph{Modified gravity without new degrees of freedom},
arXiv:0812.3200 [gr-qc], 2008.

\bibitem{gover}
A.~R. Gover, P. Somberg and V. Sou\v{c}ek,
\emph{Yang--Mills detour complexes and conformal geometry},
Comm. Math. Phys. \textbf{278} (2008), 307--327.

\bibitem{halpern}
M.~B. Halpern,
\emph{Field-strength formulation of quantum chromodynamics},
Phys. Rev. D \textbf{16} (1977), 1798--1801.

\bibitem{halpernselfdual}
M.~B. Halpern,
\emph{Gauge-invariant formulation of the self-dual sector},
Phys. Rev. D \textbf{16} (1977), 3515--3519.

\bibitem{plebanski}
J.~F. Pleba\'nski,
\emph{On the separation of Einsteinian substructures},
J. Math. Phys. \textbf{18} (1977), 2511--2520.

\bibitem{schoen}
R. Schoen,
\emph{Conformal deformation of a Riemannian metric to constant scalar curvature},
J. Differential Geom. \textbf{20} (1984), 479--495.

\bibitem{taubes}
C.~H. Taubes,
\emph{Self-dual Yang--Mills connections on non-self-dual $4$-manifolds},
J. Differential Geom. \textbf{17} (1982), 139--170.

\bibitem{trudinger}
N.~S. Trudinger,
\emph{Remarks concerning the conformal deformation of Riemannian structures on compact manifolds},
Ann. Scuola Norm. Sup. Pisa Cl. Sci. (3) \textbf{22} (1968), 265--274.

\bibitem{uhlenbeckLp}
K.~K. Uhlenbeck,
\emph{Connections with $L^p$ bounds on curvature},
Comm. Math. Phys. \textbf{83} (1982), 31--42.

\bibitem{uhlenbeckrem}
K.~K. Uhlenbeck,
\emph{Removable singularities in Yang--Mills fields},
Comm. Math. Phys. \textbf{83} (1982), 11--29.

\bibitem{yamabe}
H. Yamabe,
\emph{On a deformation of Riemannian structures on compact manifolds},
Osaka Math. J. \textbf{12} (1960), 21--37.

\bibitem{yangmills}
C.~N. Yang and R.~L. Mills,
\emph{Conservation of isotopic spin and isotopic gauge invariance},
Phys. Rev. \textbf{96} (1954), 191--195.

\end{thebibliography}
\end{document}